\documentclass[%
pdflatex,%
sn-mathphys-num%
]{sn-jnl}

\usepackage{amsmath}
\usepackage{amssymb}
\usepackage{amsthm}
\usepackage{bbm}
\usepackage[capitalise]{cleveref}
\usepackage{esdiff}
\usepackage{forest}
\usepackage{mathtools}
\usepackage{mleftright}
\usepackage{siunitx}
\usepackage{xparse}

\newtheorem{theorem}{Theorem}
\newtheorem{remark}{Remark}
\newtheorem{lemma}{Lemma}
\newtheorem{definition}{Definition}
\newtheorem{assumption}{Assumption}
\newtheorem{corollary}{Corollary}
\newtheorem{example}{Example}

\crefname{equation}{}{}
\Crefname{equation}{Equation}{Equations}

\newcommand{\abs}[1]{\left\lvert#1\right\rvert}
\newcommand{\norm}[1]{\left\lVert#1\right\rVert}

\DeclareMathOperator{\diag}{diag}
\DeclareMathOperator{\blkdiag}{blockdiag}

\newcommand{\R}{\mathbb{R}}
\newcommand{\C}{\mathbb{C}}
\newcommand{\N}{\mathbb{N}}
\renewcommand{\Re}{\operatorname{Re}}
\newcommand{\nvar}{N}
\newcommand{\timevar}{t}
\newcommand{\ts}{\timevar_{0}}
\newcommand{\tf}{\timevar_{\mathrm{f}}}
\newcommand{\dt}{h}
\newcommand{\one}{\mathbbm{1}}
\newcommand{\order}[1]{\ensuremath{\mathcal{O}\mleft(#1\mright)}}
\newcommand{\psl}{p_{\text{SL}}}
\newcommand{\so}[1]{\gamma_{#1}}
\newcommand{\qo}[1]{\widehat{\gamma}_{#1}}
\newcommand{\Vs}{V}
\DeclareDocumentCommand{\oc}{s O{\treevar} O{Z} O{t_0}}{\IfBooleanTF{#1}{\Psi}{\psi}_{#2}(#3, #4)}

\newcommand{\tp}[1]{\mathbf{#1}}

\DeclareDocumentCommand{\dY}{O{n} o}{\Delta Y_{#1}\IfValueT{#2}{^{\{#2\}}}}
\DeclareDocumentCommand{\dy}{O{n} o}{\Delta y_{#1}\IfValueT{#2}{^{\{#2\}}}}
\DeclareDocumentCommand{\dg}{O{n} o}{\Delta g_{#1}\IfValueT{#2}{^{\{#2\}}}}
\DeclareDocumentCommand{\G}{O{\ell} O{k}}{\tp{g}^{\, (#1; #2)}}
\DeclareDocumentCommand{\gk}{O{\ell} O{k}}{g^{(#1; #2)}}

\DeclareDocumentCommand{\vy}{O{n} o}{v_{#1}\IfValueT{#2}{^{\{#2\}}}}
\DeclareDocumentCommand{\ac}{O{n} o}{\alpha_{#1}\IfValueT{#2}{^{\{#2\}}}}
\DeclareDocumentCommand{\bc}{O{n} o}{\beta_{#1}\IfValueT{#2}{^{\{#2\}}}}
\newcommand{\comcoef}{\zeta}
\newcommand{\bimap}{W}
\newcommand{\scalvar}{\lambda}

\newcommand{\treevar}{\tau}

\newcommand{\treeroot}{\tau_0}
\newcommand{\treeorder}[1]{\abs{#1}}

\newcommand{\btree}[1]{\Forest{btree #1}}
\forestset{
    btree/.style={
        for tree={
            fill=black,
            circle,
            grow=north,
            minimum width=5pt,
            minimum height=5pt,
            inner sep=1pt,
            outer sep=0pt,
            s sep=3pt,
            l=3pt,
            l sep=3pt,
            if n children=0{}{
                if level=0{baseline}{}
            },
            if n children=1{
                if={mod(level(),2)==0}{
                    grow=60
                }{
                    grow=120
                }
            }{}
        }
    }
}

\newcommand{\revision}[1]{#1}

\begin{document}

\title[A Stiff Order Condition Theory for Runge--Kutta Methods Applied to Semilinear ODEs]{A Stiff Order Condition Theory for Runge--Kutta Methods Applied to Semilinear ODEs}

\author*[1]{\fnm{Steven B.} \sur{Roberts}}\email{roberts115@llnl.gov}

\author[2]{\fnm{David} \sur{Shirokoff}}\email{david.g.shirokoff@njit.edu}

\author[3]{\fnm{Abhijit} \sur{Biswas}}\email{abhijit@iitk.ac.in}

\author[4]{\fnm{Benjamin} \sur{Seibold}}\email{seibold@temple.edu}

\affil*[1]{\orgdiv{Center for Applied Scientific Computing}, \orgname{Lawrence Livermore National Laboratory}, \orgaddress{\street{7000 East Ave}, \city{Livermore}, \postcode{94550}, \state{California}, \country{United States}}}

\affil[2]{\orgdiv{Department of Mathematical Sciences}, \orgname{New Jersey Institute of Technology}, \orgaddress{\street{154 Summit Street}, \city{Newark}, \postcode{07102}, \state{New Jersey}, \country{United States}}}

\affil[3]{\orgdiv{Department of Mathematics \& Statistics}, \orgname{Indian Institute of Technology Kanpur}, \orgaddress{\street{Kalyanpur}, \city{Kanpur}, \postcode{208016}, \state{Uttar Pradesh}, \country{India}}}

\affil[4]{\orgdiv{Department of Mathematics}, \orgname{Temple University}, \orgaddress{\street{1805 N. Broad Street}, \city{Philadelphia}, \postcode{19122}, \state{Pennsylvania}, \country{United States}}}

\abstract{%
Classical convergence theory of Runge--Kutta methods assumes that the time step is small relative to the Lipschitz constant of the ordinary differential equation (ODE).
For stiff problems, that assumption is often violated, and a problematic degradation in accuracy, known as order reduction, can arise.
Methods with high stage order, e.g., Gauss--Legendre and Radau, are known to avoid order reduction, but they must be fully implicit.
For the broad class of semilinear ODEs, which consist of a stiff linear term and non-stiff nonlinear term, we show that weaker conditions suffice.
Our new semilinear order conditions are formulated in terms of orthogonality relations and can be enumerated by rooted trees.
Finally, we prove global error bounds that hold uniformly with respect to stiffness of the linear term.}



\maketitle

\section{Introduction}
This paper focuses on ordinary differential equations (ODEs) of the semilinear form
\begin{equation} \label{eq:ODE}
    y'(\timevar) = f(y(\timevar)) = J y(\timevar) + g(y(\timevar)),
    \quad
    y(\ts) = y_0,
    \quad
    \timevar \in [\ts, \tf],
\end{equation}
where $y(\timevar) \in \R^{\nvar}$.  The linear term $J y$ may be stiff and involves a constant coefficient matrix $J \in \R^{\nvar \times \nvar}$.  The term $g(y(\timevar))$ is assumed to be non-stiff but can be nonlinear. \revision{Problems} of the form \cref{eq:ODE} frequently arise from method-of-lines semi-discretizations of initial boundary value problems where $J$ is a discrete approximation to the highest order spatial derivative operator.

We consider the Runge--Kutta family of one-step methods for the numerical integration of \cref{eq:ODE}.
For a general right-hand side function $f$, they are given by
\begin{subequations} \label{eq:RK}
    \begin{align}
        \label{eq:RK:stages}
        Y_{n,i} &= y_n + \dt \sum_{j=1}^{s} a_{i,j} f(Y_{n,j}), \qquad i = 1, \dots, s, \\
        \label{eq:RK:step}
        y_{n+1} &= y_n + \dt \sum_{i=1}^{s} b_i f(Y_{n,i}),
    \end{align}
\end{subequations}
where $\dt$ is the time step on the uniform time grid $\timevar_n = \ts + n \dt$, the intermediate stages are denoted by $Y_{n,i}$, and $y_n$ is an approximation to the exact solution $y(\timevar_n)$.
Associated with this method are the $s \times s$ matrix $A = (a_{i,j})_{i,j=1}^s$, the weight vector $b = (b_i)_{i=1}^s$, 
and the abscissa vector $c = (c_i)_{i=1}^s$, herein defined by the standard row simplifying assumption $c = A \one$, where $\one = [1, \dots, 1]^T \in \R^{s}$.

\subsection{Order reduction and related work}
Unfortunately, in the presence of stiffness, Runge--Kutta methods may converge at a rate less than their classical order, denoted throughout as $p$. This order reduction phenomenon, demonstrated by Prothero and Robinson \cite{prothero1974stability}, can be understood via B-convergence theory \cite{frank1981concept}, \cite[p.~201]{dekker1984stability}, \cite[p.~219]{scholz1989order}, which provides global error bounds that hold uniformly with respect to stiffness.
For Runge--Kutta methods applied to nonlinear problems, B-convergence theory relies heavily on the following simplifying assumptions (see \cite[Lemma~2.1]{frank1985order}, \cite[Theorem~2.2]{burrage1987order}, and \cite[Section~IV.15]{hairer1996solving}):
\begin{subequations} \label{eq:simplifying_assumptions}
    \begin{alignat}{4}
        B(q_1):& \quad & b^T c^{k-1} &= \frac{1}{k}, &\quad k &= 1, \dots, q_1, \\
        C(q_2):& \quad & A c^{k-1} &= \frac{c^k}{k}, &\quad k &= 1, \dots, q_2.
    \end{alignat}
\end{subequations}
Here $c^k = [c_1^k, \ldots, c_s^k]^T$.  The minimum of $q_1$ and $q_2$ for which \cref{eq:simplifying_assumptions} holds is known as the \textit{stage order}. While fully implicit Runge--Kutta methods like Gauss--Legendre and Radau can attain high stage order, they have a high computational cost.  B-convergence results are severely limited for more computationally efficient explicit and diagonally implicit methods because their maximum stage order is one and two, respectively.

In one of the first works investigating B-convergence of Runge--Kutta methods specifically for semilinear problems, Burrage, Hundsdorfer, and Verwer showed that stringent nonlinear stability requirements for generic, nonlinear analysis are not needed; instead, weaker linear stability conditions suffice \cite{burrage1986study}. Extensions to semilinear problems in which the linear term is time-dependent were explored in \cite{auzinger1992extension,calvo2000runge}. The B-convergence results in all of these works apply only to high stage order methods. Strehmel and Weiner use a sharper per-stage simplifying assumption in \cite{strehmel1987b} to combat order reduction.
Skvortsov derived order conditions for nonlinear generalizations of the Prothero--Robinson problem \cite{skvortsov2003accuracy,skvortsov2010model}. Similar techniques have been extended to differential-algebraic equations with a focus on the incompressible Navier--Stokes equations \cite{CaiWanKareem2025}.

Outside the class of Runge--Kutta schemes, sharp error analysis for semilinear problems has been explored for exponential integrators \cite{hochbruck2005explicit,luan2013exponential,hochbruck2020convergence}, splitting methods \cite{hansen2016high,einkemmer2015overcoming,einkemmer2016overcoming}, Rosenbrock methods \cite{lubich1995linearly}, and linear multistep methods \cite[Section~V.8]{hairer1996solving}.

\subsection{Novelty and relevance of this work}
In this paper, we introduce a novel B-convergence error analysis that yields sharp order conditions for Runge--Kutta schemes when applied to stiff, semilinear ODEs \cref{eq:ODE}. Notably, our semilinear order conditions are weaker than high stage order and can be satisfied by methods that are not fully implicit. The new order conditions are posed in terms of rational functions of an auxiliary variable, similar to the weak stage order conditions for linear equations (see \cite{ketcheson2020dirk} and \cref{rem:WSO}), however, they are in one-to-one correspondence with rooted trees, analogous to those established by Albrecht \cite{albrecht1987new,albrecht1996runge} for classical, non-stiff order conditions.

In addition to providing the theoretical foundation for novel Runge--Kutta methods that overcome order reduction, the theory developed herein also\revision{:}
\begin{enumerate}
    \item[(i)] rationalizes why the weak stage order conditions (for linear problems) yield B-convergence up to order four for semilinear problems; and
    \item[(ii)] enables the construction of efficient diagonally implicit Runge--Kutta methods that mitigate order reduction for semilinear problems. New schemes will be presented in an upcoming companion paper.
\end{enumerate}
This work on implicit Runge--Kutta methods is a key stepping stone towards semi-implicit approaches like implicit-explicit (ImEx) Runge--Kutta methods. Some stiff order condition results exist for ImEx methods that are weaker than stage order (for instance, \cite{BoscarinoRusso2009, HuShu2025}), but no sharp theory. By means of this pathway, the current work can also be viewed as a parallel to the development of stiff order conditions for exponential Runge--Kutta methods applied to semilinear problems \cite{luan2013exponential}---the success of which led to the construction of high order methods \cite{LuanOstermann2014}. 

This paper is organized as follows.
\Cref{sec:math_foundations} provides the necessary mathematical background and assumptions.
\cref{sec:error_analysis} contains our analysis of the local truncation error and order condition theory.
\Cref{sec:stability_convergence} provides bounds on the global error that are uniform with respect to stiffness.
Our concluding remarks are found in \cref{sec:conclusion}.

\section{Mathematical foundations}
\label{sec:math_foundations}
This section introduces the problem assumptions, notation and background used throughout the paper.

\subsection{Vector space background}\label{subsec:vectorspaceBackground}
Throughout, we write the Runge--Kutta method \cref{eq:RK} applied to \cref{eq:ODE} in compact form as
\begin{subequations} \label{eq:RK_kron}
    \begin{alignat}{3}
        \label{eq:RK_kron:stages}
        Y_n &=& \one \otimes y_{n} &+ (A \otimes Z) Y_n &&+ \dt (A \otimes I) \, \tp{g}(Y_n), \\
        \label{eq:RK_kron:step}
        y_{n+1} &=& y_{n} &+ (b^T \otimes Z) Y_n &&+ \dt (b^T \otimes I) \, \tp{g}(Y_n),
    \end{alignat}
\end{subequations}
where $Z \coloneqq \dt J$ and 
\begin{subequations}
    \begin{alignat*}{2}
        Y_n &\coloneqq \left[
            Y_{n,1}^T, \dots, Y_{n,s}^T
        \right]^T, & \qquad 
        \tp{g}(Y_n) &\coloneqq \left[
            g(Y_{n,1})^T, \dots, g(Y_{n,s})^T
        \right]^T \; \in \R^s \otimes \R^\nvar \cong \R^{s\nvar}.
    \end{alignat*}
\end{subequations}
We use bold functions to \revision{denote} the concatenation of function evaluations for all stages. For the vector space $\R^s$ we use the standard Euclidean inner product; for $\R^\nvar$ we allow any inner product $\langle \cdot, \cdot\rangle$. These then define an inner product on $\R^s \otimes \R^\nvar$ as
\begin{equation}\label{Eq:InnerProd}
    \langle X, U \rangle = \sum_{i = 1}^s \langle x_i, u_i \rangle
    \ \ \textrm{where} \ \
    X = \begin{bmatrix} x_1^T, \ldots, x_s^T \end{bmatrix}^T,\,
    U = \begin{bmatrix} u_1^T, \ldots, u_s^T \end{bmatrix}^T \in \R^s \otimes \R^\nvar,
\end{equation}
with $x_i, u_i \in \R^\nvar$ for $i = 1, \ldots, s$, and associated matrix 2-norms, e.g., $\norm{A}$ or $\norm{A \otimes Z}$. 

For each $x \in \R^\nvar$, the $k$-th derivative of $g$
is a multilinear map of the vectors $u_1, \ldots, u_k \in \R^\nvar$ denoted by
\begin{equation*}
    g^{(k)}(x)(u_1, u_2, \ldots, u_k) \in \R^\nvar \qquad (\textrm{or} \; g', g'' \; \textrm{when} \; k = 1, 2).
\end{equation*}
Evaluating $g^{(k)}$ along the ODE trajectory $y(t)$ yields the family of $k$-linear maps 
\begin{equation*}
    \gk(t)(u_1, u_2, \ldots, u_k) \coloneqq \diff[\ell]{}{t} g^{(k)}( y(t) )(u_1, u_2, \ldots, u_k) \qquad \revision{(\ell  = 0, 1, 2, \dots)}.
\end{equation*}
The definitions also extend to the vector-valued function $\tp{g}$ as follows. If
\begin{equation*}
    X = 
    \begin{bmatrix}
        x_1 \\
        \vdots \\
        x_s
    \end{bmatrix} \in \R^s \otimes \R^\nvar \quad \textrm{and} \quad
    U_j = \begin{bmatrix} 
        u_{j,1} \\
        \vdots \\
        u_{j, s}
    \end{bmatrix} \in \R^s \otimes \R^\nvar \qquad (j =1, \ldots, k)
\end{equation*}
are a set of vectors with block components $x_i, u_{j, i} \in \R^\nvar$ ($i = 1, \ldots, s$), then 
\begin{equation*}
        \tp{g}^{\,(k)}( X ) (U_1, \dots, U_k) \coloneqq 
        \begin{bmatrix}
            g^{(k)}(x_1) (u_{1,1}, \dots, u_{k,1}) \\
            \vdots \\
            g^{(k)}(x_s) (u_{1,s}, \dots, u_{k,s})
        \end{bmatrix} \in \R^s \otimes \R^\nvar.
\end{equation*}
In the special case when $X = \one \otimes y(t)$ (where $y(t)$ is the solution of the ODE), we denote the time derivative of $\tp{g}^{\, (k)}$ by
\begin{equation*}
    \G(t)(U_1, \dots, U_k) \coloneqq 
    \diff[\ell]{}{t} \tp{g}^{\, (k)}\big( \one \otimes y(t)\big) (U_1, \dots, U_k)
    =
    \begin{bmatrix}
            \gk(t)(u_{1,1}, \dots, u_{k,1}) \\
            \vdots \\
            \gk(t)(u_{1,s}, \dots, u_{k,s})
        \end{bmatrix}.        
\end{equation*}

\subsection{Problem assumptions}\label{subsec:assumptions}
Throughout, we make the following assumptions on the ODE. 
\begin{assumption} \label{assump:ODE}
    We assume \cref{eq:ODE} satisfies the following properties:
    \begin{subequations} \label{eq:ODE_assumptions}
        \begin{enumerate}
            \item $J \in \R^{\nvar \times \nvar}$ has nonpositive logarithmic 2-norm: 
            \begin{equation} \label{eq:ODE_assumptions:linear}
                \mu(J)
                \coloneqq \max_{\norm{x} = 1} \langle x, J x \rangle \leq 0.
            \end{equation}
            \item $g$ is Lipschitz continuous, that is, there exists an $L > 0$ such that
            \begin{equation} \label{eq:ODE_assumptions:nonlinear}
                \norm{g(y) - g(z)} \leq L \norm{y-z}
                \quad \forall \,  y, z \in \R^N.
            \end{equation}

            \item All partial derivatives of $g$ up to order $r$ exist and are continuous, and $y$ is $r+1$ times continuously differentiable. Furthermore, there exists a constant $M$ such that
            \begin{equation}
                \label{eq:ODE_assumptions:diff}
                \begin{alignedat}{3}
                    \norm{y^{(k)}(\timevar)} &\leq M
                    &\quad &\forall \, \timevar \in [\ts, \tf],
                    &\quad k &= 1, \dots, r+1, \\
                    \norm{g^{(k)}(y)} &\leq M
                    &\quad &\forall \, y \in \R^N,
                    &\quad k &= 1, \dots, r.
                \end{alignedat}
            \end{equation}
        \end{enumerate}
    \end{subequations}
\end{assumption}

\noindent Several remarks are in order:
\begin{itemize}
    \item As is common (e.g., \cite{burrage1986study, calvo2000runge}), we assume a one-sided Lipschitz condition \cref{eq:ODE_assumptions:linear} which allows the eigenvalues of $J$ to extend arbitrarily far into the left-half plane.
    \item If \cref{eq:ODE_assumptions:nonlinear,eq:ODE_assumptions:diff} are replaced with local bounds on $g$ in a tubular neighborhood of the solution $y(t)$, \cref{thm:MainLTE,thm:convergence} still hold with an additional restriction on $\dt$. 
    \item If $\mu(J) > 0$, then the ODE can be recast to satisfy the problem assumptions with $g \rightarrow g + \mu(J) y$, $J \rightarrow J - \mu(J)I$, $L \rightarrow L + \mu(J)$, etc., as done in \cite[p.~617]{calvo2000runge}.
    \item The techniques in this paper also generalize to a broader class of semilinear problems of the form $y' = J y + g(y) + r(t)$. Versions of \cref{thm:MainLTE,thm:convergence} hold with an error constant $D$ that is not only uniform in $\mu(J) (\leq 0)$ but also in $\norm{r}$.
\end{itemize}

Complementing the regularity and boundedness assumptions are the following stability conditions. A Runge--Kutta scheme is A-stable if $\abs{R(z)} \leq 1$ for all $z \in \C^{-}$ where
\begin{equation*}
    \C^{-} \coloneqq \{z \in \C : \Re(z) \leq 0 \} \qquad
    \text{and} \qquad 
    R(z) \coloneqq 1 + z b^T (I - z A)^{-1} \one.
\end{equation*}
We also use the following less-common notions of linear stability.

\begin{definition}[{\cite[Definition~3.1]{burrage1986study}}]
    The Runge--Kutta method \cref{eq:RK} is ASI-stable if $I - z A$ is non-singular for all $z \in \C^-$ and $(I - z A)^{-1}$ is uniformly bounded for $z \in \C^-$.
\end{definition}

\begin{definition}[{\cite[Definition~3.2]{burrage1986study}}]
    The Runge--Kutta method \cref{eq:RK} is AS-stable if $I - z A$ is non-singular for all $z \in \C^-$ and $z b^T (I - z A)^{-1}$ is uniformly bounded for $z \in \C^-$.
\end{definition}

Runge--Kutta methods with a lower triangular Butcher matrix $A$ are referred to as
diagonally implicit Runge--Kutta (DIRK) methods. If all diagonal entries, $a_{i,i}$, of a DIRK method are positive, it is both AS- and ASI-stable (by \cite[Lemmas~4.3 and 4.4]{burrage1986study}).
Stiffly \revision{accurate} \cite[p.~92]{hairer1996solving} DIRK methods with an explicit first stage and $a_{i,i}$ positive for $i = 2, \dots, s$ are also AS- and ASI-stable.

The AS- and ASI-stability properties have natural matrix-valued extensions.
\begin{lemma} \label{lem:AS_ASI_stability_bounds}
    For an AS- and ASI-stable Runge--Kutta method, the matrix $I - A \otimes Z$ is non-singular and the matrix norms of 
    \begin{equation*}
        (I - A \otimes Z)^{-1}   \qquad \textrm{and} \qquad (b^T \otimes Z) (I - A \otimes Z)^{-1} 
    \end{equation*}
    are uniformly bounded for $Z \in \R^{\nvar \times \nvar}$ such that $\mu(Z) \leq 0$.
\end{lemma}

The proof follows from the matrix version of a theorem by von Neumann.
\begin{theorem}[Nevanlinna, {\cite[Corollary 3]{nevanlinna1985matrix}}]\label{Thm:Nevanlinna}
    Suppose $D(z) = (d_{i,j}(z) )$ is an $m\times n$ matrix whose elements $d_{i,j}(z)$ are rational functions of a complex variable $z$. Then 
    \begin{align*}
        \norm{D(Z)} \leq \sup_{z \in \C^{-}} \norm{D(z)}     
    \end{align*}    
    holds for any $Z \in \R^{\nvar \times \nvar}$ such that $\mu(Z) \leq 0$. 
    In $\norm{D(z)}$, the norm is the matrix 2-norm, and in $\norm{D(Z)}$, the norm is induced by the inner product \cref{Eq:InnerProd}.
\end{theorem}
\begin{proof}[Proof of {\cref{lem:AS_ASI_stability_bounds}}]
    If $\lambda(A)$ is any eigenvalue of $A$, then by AS- and ASI-stability, 
    \begin{equation*}
        1 - z \lambda(A) \neq 0   \quad \forall z \in \mathbb{C}^-.
    \end{equation*}
    The eigenvalues of $I - A \otimes Z$ take the form $1 - \lambda(A) \lambda(Z)$, where the eigenvalues of $Z$ satisfy $\lambda(Z) \in \C^{-}$ because $\mu(Z) \leq 0$.  Therefore, no eigenvalue of $I - A \otimes Z$ can be zero and the matrix must be non-singular. Setting $D(z)$ to be $(I - z A)^{-1}$ and $z b^T (I - z A)^{-1}$ in \cref{Thm:Nevanlinna} proves the result.
\end{proof}

\subsection{Background on the local truncation error}\label{subsec:BackgroundLTE}
Following \cite{burrage1986study}, an equation for the local truncation error is obtained by observing the exact solution can be viewed as satisfying \cref{eq:RK_kron} with a residual defect $\Delta_0$ and $\delta_0$:
\begin{subequations}\label{eq:RK_stages_exact_sol}
\begin{alignat}{5}     
        \tp{y}(\ts) &=& \one \otimes y(\ts)&\, + \,&\big(A \otimes Z\big)\, \tp{y}(\ts) &\,+\, &\dt \big(A \otimes I\big)\, \tp{g}\big(\tp{y}(\ts)\big)&\, +\, &\Delta_0, &\\
        y(\timevar_{1}) &=& y(\ts)&\, + \, & \big(b^T \otimes Z\big)\, \tp{y}(\ts)&\, + \,&\dt \big(b^T \otimes I\big)\, \tp{g}\big(\tp{y}(\ts)\big) & \,+\, &\delta_0. &    
\end{alignat}    
\end{subequations}
Here,
\begin{equation}\label{Eq:ExactSol}
    \tp{y}(\ts) \coloneqq \big[y(\ts + c_1 \dt)^T, \dots, y(\ts + c_s \dt)^T\big]^T \in \mathbb{R}^s \otimes \mathbb{R}^\nvar
\end{equation}
is the exact solution evaluated via the abscissae. The defects are given by 
\begin{equation}\label{Eq:defects}
    \Delta_0 \coloneqq \sum_{i=1}^{r} \dt^i \so{i} \otimes y^{(i)}(\ts) + \order{\dt^{r + 1}}
    \quad \text{and} \quad
    \delta_0 \coloneqq \sum_{i=1}^{r} \dt^i \qo{i} y^{(i)}(\ts) + \order{\dt^{r + 1}},
\end{equation}
where
\begin{equation*}
    \qo{\ell} \coloneqq \frac{1}{\ell!} - \frac{b^T c^{\ell-1}}{(\ell-1)!} 
    \quad \textrm{and} \quad 
    \so{\ell} \coloneqq \frac{c^{\ell}}{\ell!} - \frac{A c^{\ell-1}}{(\ell-1)!} 
\end{equation*}
are scaled residuals of the $B(\ell)$ and $C(\ell)$ simplifying assumptions \cref{eq:simplifying_assumptions}, respectively. The constants in the $\order{\dt^{r+1}}$ depend only on $y^{(r+1)}$. 
Note that the formulas for the defects \cref{Eq:defects} (in terms of the exact solution) can be obtained by substituting the ODE equation into \cref{eq:RK_stages_exact_sol} and expanding via Taylor series.

Denote the errors between the numerical and exact solutions by
\begin{equation} \label{eq:global_err_def}
    \begin{split}
        \dy &\coloneqq y(\timevar_n) - y_n, \\
        \dY &\coloneqq
        \tp{y}(\timevar_n) - Y_n \\
        &= \left[
            \big(y(\timevar_n + c_1 \dt) - Y_{n,1}\big)^T, \dots, \big(y(\timevar_n + c_s \dt) - Y_{n,s}\big)^T
        \right]^T, \\ 
        \dg &\coloneqq
        \tp{g}(\tp{y}(\timevar_n)) - \tp{g}(Y_n) \\
        &= \left[
           \big(g(y(\timevar_n + c_1 \dt)) - g(Y_{n,1}) \big)^T, \dots, \big(g(y(\timevar_n + c_s \dt)) - g(Y_{n,s}) \big)^T
        \right]^T.
    \end{split}
\end{equation}
Subtracting \cref{eq:RK_kron} from \cref{eq:RK_stages_exact_sol} yields the recursion relation for the errors:
\begin{subequations}\label{eq:error_recur}    
    \begin{alignat}{4}\label{eq:error_recur1}
        \dY[n] &=& \one \otimes \dy[n] &+& \big(A \otimes Z\big) \dY[n] &+& \dt \big(A \otimes I\big) \dg[n] &+ \Delta_0, \\ \label{eq:error_recur2}
        \dy[n+1] &=& \dy[n] &+& \big(b^T \otimes Z\big) \dY[n] &+& \dt \big(b^T \otimes I\big) \dg[n] &+ \delta_0.
    \end{alignat}
\end{subequations}
When \revision{$\dy[n] = 0$, we refer to $\dy[n+1]$ as the local truncation error (LTE) at $t_{n+1}$}.

A tacit assumption up to this point is that the solution $Y_n$ in \cref{eq:RK_kron:stages} exists.
Indeed, the following theorem is significant as it shows the solution exists for a range of $\dt$ values that are independent of the size of $Z$ (i.e., the stiffness).

\begin{theorem}[Calvo, Gonz{\'a}lez-Pinto, Montijano {\cite[Section~4.3]{calvo2000runge}}] \label{thm:exist_unique}
    Consider an ASI-stable Runge--Kutta method \cref{eq:RK} used to solve the ODE \cref{eq:ODE} under \cref{assump:ODE}.
    There exists a positive constant $\widetilde{\dt}$, dependent only on $L$ and the method coefficients, such that for all $\dt \in [0, \widetilde{\dt})$ there exists a unique solution to the stage equations \cref{eq:RK:stages}.
\end{theorem}
Theorem \cite[Section~4.3]{calvo2000runge} actually proves a stronger result allowing for a time-dependent $J$.  
Here, we view $Y_n$ and $y_n$ (and by extension $\dY[n]$) as a function of two independent variables $h$ and $Z$. When the fixed point iteration in \cite[Section~4.3]{calvo2000runge} for $Y_n$ is combined with the standard contraction mapping proof of the implicit/inverse function theorem (see \cite[Chapter~1.3]{TaylorPDE1}), \cref{assump:ODE} implies $Y_n$ and $y_n$ are $r$ times continuously differentiable function of $h$ on $[0, \widetilde{\dt})$.

\section{Local truncation error analysis} \label{sec:error_analysis}
This section develops the main theoretical result of the paper: the derivation of the semilinear order conditions.
We show they are in one-to-one correspondence with rooted trees, and this graphical interpretation helps to identify redundant conditions.

\subsection{Power series expansions of the local truncation error}\label{subsec:LTERecursion}
The classical (non-stiff) Runge--Kutta order conditions are most often derived through a B-series expansion for the LTE. An alternative approach, originally proposed by Albrecht \cite{albrecht1987new,albrecht1996runge} for Runge--Kutta methods, introduces recursively-generated orthogonality conditions for the LTE.  

In this subsection we expand Albrecht's approach to the semilinear setting. Notably, the power series expansions for the LTE depends on $\tp{g}$, $y$, and expressions that are uniformly bounded in $Z$ (i.e., the stiffness).

\begin{lemma}[Recursive formula for LTE]\label{thm:LTE}
    Suppose the Runge--Kutta scheme \cref{eq:RK} is AS- and ASI-stable, \cref{assump:ODE} holds, and the time step satisfies $\dt < \widetilde{\dt}$ where $\widetilde{\dt}$ is given by \cref{thm:exist_unique}.  When applied to \cref{eq:ODE}, the errors \cref{eq:global_err_def} for the first step admit the series
    \begin{equation} \label{eq:LTE_series}
        \begin{split}
            \dY[0] &= \sum_{i=1}^{r} \dY[0][i] \dt^i + \varepsilon_{Y}(\dt),\\
            \dy[1] &= \sum_{i=1}^{r} \dy[1][i] \dt^i + \varepsilon_{y}(\dt),\\
            \dg[0] &= \sum_{i=\revision{0}}^{r-1} \dg[0][i] \dt^i + \varepsilon_{g}(\dt),
        \end{split}
    \end{equation}
    where $\norm{\varepsilon_{Y}},\norm{\varepsilon_{y}} \leq D \dt^{r+1}$ and $\norm{\varepsilon_{g}} \leq D \dt^{r}$ hold with a constant $D$ depending only on $M$, $L$ and method coefficients (but not $Z$). The coefficients are defined recursively by 
    \begin{subequations} \label{eq:LTE_coeffs}
        \begin{align}\label{eq:LTE_coeffs:Y}
            \dY[0][i] &= \big(I - A \otimes Z\big)^{-1} \left( \so{i} \otimes y^{(i)}(\ts) \right) + \big(A \otimes I\big) \big(I - A \otimes Z\big)^{-1} \dg[0][i-1], \\ \label{eq:LTE_coeffs:y}
            \begin{split}
            \dy[1][i] &= \big(b^T \otimes Z\big) \big(I - A \otimes Z\big)^{-1} \left( \so{i} \otimes y^{(i)}(\ts) \right) \\
            & \quad + \big(b^T \otimes I\big) \big(I - A \otimes Z\big)^{-1} \dg[0][i-1] 
                +\qo{i} y^{(i)}(\ts) ,
            \end{split} \\
                \label{eq:LTE_coeffs:g}
            \dg[0][i] &= \sum_{\ell = 0}^{i - 1} \ \sum_{k=1}^{i - \ell} \
            \sum_{\substack{ m_1 + \cdots + \\ m_k = i - \ell } }
            \frac{(-1)^{k+1}}{k! \ell!} (C^{\ell} \otimes I) \G(\ts) \mleft( \dY[0][m_1], \dots, \dY[0][m_{k}] \mright).
        \end{align}
    \end{subequations}
    Here $C = \diag(c)$, while the subscript $m_1 + \ldots + m_k = i -\ell$ denotes a summation over all positive $k$-tuples $(m_1, \ldots, m_k) \in \mathbb{Z}_+^{k}$ whose sum is $i - \ell$.
\end{lemma}

\begin{remark}
    Note that \revision{$\dg[0][0] = 0$, and} terms in \cref{eq:LTE_coeffs:g} for which one of the $m_j = 1$ vanish since $c = A \one$ implies $\dY[0][1] = 0$.
\end{remark}

\begin{remark} \label{rem:classical_special_case}
    By setting $Z=0$ in \cref{eq:LTE_coeffs}, we recover the classical, non-stiff error expansion of Albrecht in \cite[Recursion~0]{albrecht1996runge}.
    Without loss of generality, Albrecht was able to derive order conditions looking at scalar ODEs.  However, for our semilinear analysis, scalar problems allow terms in \cref{eq:LTE_coeffs} to commute, for example $I \otimes g'(y(t))$ and $(I - A \otimes Z)^{-1}$, and leads to an incomplete set of order conditions.
    This discrepancy starts at order three terms and is why we require Kronecker products.
\end{remark}

\begin{proof}
    Since the initial condition is exact, $\dy[0] = 0$, \cref{eq:error_recur} can be manipulated into 
    \begin{subequations} \label{eq:LTE}
        \begin{align}\label{eq:LTEa}
            \dY[0] &= \dt \big(A \otimes I\big) \big(I - A \otimes Z\big)^{-1} \dg[0] + \big(I - A \otimes Z\big)^{-1} \Delta_0, \\ \label{eq:LTEb}
            \dy[1] &= \dt \big(b^T \otimes I\big) \big(I - A \otimes Z\big)^{-1} \dg[0] + \big(b^T \otimes Z\big) \big(I - A \otimes Z\big)^{-1} \Delta_0 + \delta_0.
        \end{align}
    \end{subequations}
    Existence and boundedness of the terms $(I - A \otimes Z)^{-1}$ and $(b^T \otimes Z) (I - A \otimes Z)^{-1}$ follow from \cref{lem:AS_ASI_stability_bounds}.  

    Since $\dY[0] = \dY[0](\dt)$ is $r$ times continuously differentiable (see \cref{subsec:BackgroundLTE}), so is $\dy[1]$; both can be expanded in a power series of the form \cref{eq:LTE_coeffs:Y}.  To estimate $\varepsilon_{Y}$, which is the $r$-th Taylor remainder, take the $r$-th derivative of \cref{eq:LTEa}, which shows
    \begin{equation}\label{Eq:varForm}
        \dY[0]^{(r)}(\dt) = \dt f_1(\dt) + f_2(\dt),
    \end{equation}    
    where $f_1$ is continuous and $f_2$ continuously differentiable. Since the matrix inverses in \cref{eq:LTE} are bounded, the norms of $\|f_1\|$, $\|f_2\|$, $\|f_2'\|$ are bounded by constants depending only on $L$, $M$, and the method coefficients (e.g., by induction on the derivatives of $\dY[0]^{(j)}$ via \cref{eq:LTEa}). Hence, $\|\dY[0]^{(r)}(h) - \dY[0]^{(r)}(0)\| \leq D' h$ for a constant $D'$ that depends only on $L$, $M$ and the method coefficients. The integral version of the Taylor remainder theorem for vector-valued functions yields
    \begin{align*}
        \| \varepsilon_{Y} \|
        &= \frac{h^{r}}{r!} \norm{ \int_0^1 \dY[0]^{(r)}\big( (1 - u^{1/r})h \big) - \dY[0]^{(r)}(0) \, \textrm{d}u \, } \\
        &\leq \frac{h^{r}}{r!} \int_0^1 \norm{ \dY[0]^{(r)}\big( (1 - u^{1/r})h \big) - \dY[0]^{(r)}(0) } \, \textrm{d}u \\
        &\leq D h^{r+1}\,, \qquad \textrm{where} \quad D = D'/r!\;, 
    \end{align*}
    where $1-u^{1/r}$ arises from a change of variables in the standard expression.
    A similar estimate holds for $\varepsilon_{y}$ and $\varepsilon_g$.

    The recurrence formulas \cref{eq:LTE_coeffs:Y} and \cref{eq:LTE_coeffs:y} follow from substituting \cref{eq:LTE_series} into \cref{eq:LTE} and matching powers of $\dt$. To obtain the expression for $\dg[0][i]$, first expand $\dg[0]$ via Taylor series as
    \begin{equation}\label{Eq:Taylor_g}
        \dg[0] =
        \tp{g}\big(\tp{y}(\ts)\big) - \tp{g}\big(\tp{y}(\ts) - \dY[0]\big) \\
        = \sum_{k=1}^{r-1} \frac{(-1)^{k+1}}{k!} \tp{g}^{\, (k)}\big(\tp{y}(\ts)\big) (\underbrace{\dY[0], \dots, \dY[0]}_{k \text{ times}}) + \order{\dt^{r}} .
    \end{equation}
    Here the constant in $\order{\cdot}$ depends on derivatives of $g$ and is bounded independent of $Z$.
    Next, Taylor-expand each $\tp{g}^{\, (k)}$ in powers of $\dt$ to obtain (suppressing the arguments of the multilinear map)
    \begin{equation}\label{eq:PowerSeries_gk}
        \tp{g}^{\, (k)}\big(\tp{y}(\ts)\big) = 
        \sum_{\ell = 0}^{r-1} \frac{\dt^{\ell}}{\ell!} \diff[\ell]{}{\dt} \tp{g}^{\, (k)}\big(\tp{y}(\ts ) \big)\Big|_{\dt = 0} + \order{ \dt^{r} }. 
    \end{equation}
    From the definition of $\tp{y}(t_0)$ in \cref{Eq:ExactSol}, each term in the series in \cref{eq:PowerSeries_gk} has the form
    \begin{equation}\label{Eq:SubTermPowerSeries_gk} 
        \diff[\ell]{}{\dt} \tp{g}^{\, (k)}\big(\tp{y}(\ts ) \big)\Big|_{\dt = 0}
        =  (C^{\ell} \otimes I) \, \G(\ts) .
    \end{equation}
    Combining \cref{Eq:Taylor_g}, \cref{eq:PowerSeries_gk}, and \cref{Eq:SubTermPowerSeries_gk} yields
    \begin{equation}\label{Eq:dgFull}
            \dg[0] = \sum_{k=1}^{r-1} \sum_{\ell=0}^{r-1-k} \frac{(-1)^{k+1} \dt^{\ell}}{k! \ell!} (C^{\ell} \otimes I) \, \G(\ts) (\underbrace{\dY[0], \dots, \dY[0]}_{k \text{ times}}) + \order{\dt^{r}}.
    \end{equation}
    Substituting the series \cref{eq:LTE_coeffs:Y} for $\dY[0]$ into $ \G(\ts)$ yields a $k$-fold sum over terms of the form
    \begin{equation} \label{eq:multi_ind_g}
        \begin{split}
            & \quad \G(\ts)(\dY[0], \dots, \dY[0]) \\
            &= \sum_{m_1 + \ldots + m_k < r} \G(\ts)\mleft(\dY[0][m_1], \dots, \dY[0][m_k]\mright) \, \dt^{m_1 + \dots + m_k} + \order{\dt^r}.
        \end{split}
    \end{equation}
    Finally, substituting \cref{eq:multi_ind_g} into \cref{Eq:dgFull} yields \cref{eq:LTE_coeffs:g} for $\dg[0][i]$ as the coefficient of $\dt^{i}$. Note that every term in \cref{eq:multi_ind_g} satisfying $m_1 + \dots + m_k = i - \ell$ appears in $\dg[0][i]$. 
\end{proof}
\cref{thm:LTE} yields a systematic, albeit tedious, algorithm to derive the LTE up to a desired order. The series coefficients of $\dy[1]$ through order three, for example, are
\begin{equation} \label{eq:LTE_order_3}
    \begin{split}
        \dy[1] &= \dt \, \qo{1} \, y'(\ts)
        + \dt^2 \,\qo{2} \, y''(\ts)
        + \dt^2 \big(b^T \otimes Z\big) \big(I - A \otimes Z\big)^{-1} \big( \so{2} \otimes y''(\ts) \big) \\
        & \quad + \dt^3 \, \qo{3} \, y^{(3)}(\ts)
        + \dt^3 \big(b^T \otimes Z\big) \big(I - A \otimes Z\big)^{-1} \big( \so{3} \otimes y^{(3)}(\ts) \big) \\
        & \quad + \dt^3 \big(b^T \otimes I\big) \big(I - A \otimes Z\big)^{-1} \big(I \otimes g'(y_0)\big)\big(I - A \otimes Z\big)^{-1} \big( \so{2} \otimes y''(\ts) \big) + \order{\dt^4}.
    \end{split}
\end{equation}
It is important to note that the differentials appearing in \cref{eq:LTE_order_3}, and more generally in \cref{eq:LTE_coeffs}, are of $y$ and $g$.  These are bounded under \cref{assump:ODE}.
The unbounded $Z$ is judiciously confined to expressions that are bounded by AS- and ASI-stability.

\subsection{Tree representation for the local truncation error}
\label{subsec:treerep}
Our next goal is to expand $\dy[1][i]$ as linearly independent combinations of differentials involving $g$ and $y$.
The expansions provide a systematic pathway to compute \emph{semilinear order conditions}, which are presented in \cref{subsec:Semilinear_OC} up to fifth order.

Our solution expansion for $\dy[1][i]$ follows the spirit of Albrecht for classical (non-stiff) order conditions \cite[Section~4]{albrecht1996runge}.
Albrecht's recursion leads to an expansion in a different set of differentials and corresponding weights than Butcher's rooted-tree-based B-series approach.

We use the sets
\begin{alignat*}{2}
    T &\coloneqq \{
        \btree{[]},
        \btree{[[]]},
        \btree{[[][]]},
        \btree{[[[]]]},
        \btree{[[][][]]},
        \btree{[[[]][]]},
        \dots
    \},
    & \qquad & \text{(rooted trees)} \\
    T_i &\coloneqq \{ \treevar \in T : \treeorder{\treevar} = i \},
    & \qquad & \text{(rooted trees with $i$ vertices)}
\end{alignat*}
where $\treeorder{\treevar}$ denotes the number of vertices in a tree.
The tree with one vertex is denoted by $\treeroot \coloneqq \btree{[]}$.
The \textit{standardized form} \cite[Section~4.1]{albrecht1996runge} of a tree is
\begin{equation*}
    \treevar = [\treeroot^{\ell} \, \treevar_1 \dots \treevar_k],
    \qquad
    \ell \geq 0, \quad k \geq 0,
    \quad
    \treevar_i \neq \treeroot \quad (i = 1,\ldots, k),
\end{equation*}
where the brackets indicate joining the subtree arguments to a shared root.
The exponent in $\treeroot^{\ell}$ is the number of terminal nodes that are children of the root node.
In standardized form, the term $\treeroot^{\ell}$ always appears first in the list of subtrees provided $\ell>0$.
For example,
\begin{equation*}
    \btree{[[][[]][][[[][][]]]]}
    = [\treeroot^2 \, \btree{[[[][][]]]} \, \btree{[[]]}]
    = [\treeroot^2 \, [\btree{[[][][]]}] \, [\treeroot]]
    = [\treeroot^2 \, [[\treeroot^3]] \, [\treeroot]].
\end{equation*}

The proof of the next theorem is deferred to \cref{app:abstractrec}.
\begin{theorem} \label{thm:LTE_trees}
    The LTE \cref{eq:LTE_series} can be expressed as
    \begin{subequations} \label{eq:LTE_tree_series}
        \begin{align}
            \label{eq:LTE_tree_series:Y}
            \dY[0] &= \sum_{\substack{\treevar \in T \\ \abs{\treevar} \leq r}} \zeta_{\treevar} \oc* \dt^{\abs{\treevar}} + \order{\dt^{r + 1}}, \\
            \label{eq:LTE_tree_series:y}
            \dy[1] &= \sum_{\substack{\treevar \in T \\ \abs{\treevar} \leq r}} \zeta_{\treevar} \oc \dt^{\abs{\treevar}} + \order{\dt^{r + 1}},
        \end{align}
    \end{subequations}
    where the constants in $\order{\cdot}$ are the same as in \cref{thm:LTE} and depend only on $M$, $L$ and the method coefficients (not on $J$ or $Z$). The term $\zeta_{\treevar}$ gives a real number depending only on $\treevar$. The functions $\oc*[\treevar][Z][t]$ and $\oc[\treevar][Z][t]$ are defined as follows:\\[.2em] 
    \begin{subequations}\label{eq:LTE_tree}     
    If $\treevar = [\treeroot^{\ell}]$ then
    \begin{align}\label{eq:LTE_tree:Stage}
        \oc*[\treevar][Z][t] &\coloneqq (I - A \otimes Z)^{-1} \left( \so{\ell+1} \otimes y^{(\ell+1)}(\timevar) \right), \\
        \oc[\treevar][Z][t] &\coloneqq \qo{\ell+1} y^{(\ell+1)}(\timevar) + (b^T \otimes Z) \oc*[\treevar][Z][\timevar].
    \end{align}
    If $\treevar = [\treeroot^{\ell} \, \treevar_1 \dots \treevar_k]$ then 
        \begin{align}\label{eq:LTE_tree:y}
            \oc*[\treevar][Z][t] &\coloneqq 
            (I - A \otimes Z)^{-1} \big((A C^{\ell}) \otimes I\big) \, \G(\timevar) \big( \oc*[\treevar_1][Z][t], \ldots, \oc*[\treevar_k][Z][t] \big), \\
            \label{eq:LTE_tree:psi}
            \oc[\treevar][Z][t] &\coloneqq (b^T \otimes I) (I - A \otimes Z)^{-1} \big(C^{\ell} \otimes I\big) \, \G(\timevar) \big( \oc*[\treevar_1][Z][t], \dots, \oc*[\treevar_k][Z][t] \big). 
    \end{align}
    \end{subequations}
\end{theorem}

Note that \cref{eq:LTE_tree} defines $\psi$ for any arbitrary set of sufficiently smooth functions $g$, $y$ and suitable matrices $Z$. For any fixed $\tau$, the function $\oc[\treevar][z][t]$ is a rational function of $z \in \C$ with time-dependent coefficients depending on $|\tau| - 1$ derivatives of $g$ and $|\tau|$ derivatives of $y$. In the next section, we formulate conditions on $(A,b,c)$ which ensure that $\oc = 0$ for any such choice $g$, $y$ and suitable matrices $Z$.

\subsection{Semilinear order conditions}\label{subsec:Semilinear_OC}
While $\oc[\treevar][Z][\ts]$ characterizes the LTE, the tensor product structure interweaves method coefficients with differentials.  Ultimately, we seek semilinear order conditions that depend only on the method coefficients.  This section establishes a recursive formula, based on trees, for the order conditions.

We use two algebraic identities. First, a consequence of the Cayley--Hamilton theorem: for any $Z$ satisfying \cref{assump:ODE}, there exists polynomials $P$ and $Q_i$ ($i = 0, \ldots, s-1$), whose coefficients depend on both $Z$ and $A$, for which
\begin{equation}\label{Eq:MatExpansion}
    (I - A \otimes Z)^{-1}
    = P(A \otimes Z)
    = \sum_{i=0}^{s-1} A^i \otimes Q_i(Z).
\end{equation}
The degrees of $P$ and $Q_i$ are bounded by $\nvar s - 1$ and $\nvar - 1$, respectively. 

The second identity is a Kronecker product result for $\G(\timevar)$, and it follows from the fact that $\G(\timevar)$ is $s$ evaluations of the same function $g$. For any set of vectors $\beta_j \in \R^s, u_j \in \R^{\nvar}$ ($j = 1, \ldots, k$), the following holds:
\begin{equation}
\label{Eq:KroneckerIdentity}
    \G(\timevar)(\beta_1 \otimes u_1, \ldots, 
    \beta_k\otimes u_k) = (\beta_1 \times \cdots \times \beta_k)
    \otimes \gk(\timevar)(u_1, \ldots, u_k).
\end{equation}
Here, the notation $\times$ is the element-wise product of vectors.

With these identities, we may extract semilinear order conditions, one for each rooted tree. The simplest setting is the ``bushy'' trees.
\begin{remark}\label{rem:WSO}
    For a bushy tree $\treevar = [\treeroot^\ell]$, with $\ell \geq 0$, we have that
    \begin{align}\nonumber 
        \oc[\treevar][Z][t]
        &= \qo{\ell+1} y^{(\ell+1)}(\timevar) + (b^T \otimes Z) (I - A \otimes Z)^{-1} \left( \so{\ell+1} \otimes y^{(\ell+1)}(\timevar) \right) \\ \label{Eq:ErrWSOPsi}
        &= \qo{\ell+1} y^{(\ell+1)}(\timevar) + \sum_{i=0}^{s - 1} (b^T A^i \so{\ell+1}) Z Q_i(Z) y^{(\ell+1)}(\timevar).
    \end{align}
    Setting the coefficients $\qo{\ell +1} = 0$ and $b^T A^i \so{\ell+1} = 0$ yields the order conditions
    \begin{subequations} \label{eq:linear_conditions}
        \begin{align}
            b^T c^{\ell} &= \frac{1}{\ell + 1}, \\ \label{eq:linear_conditions:WSO}
            b^T A^i \left(\frac{c^{\ell+1}}{\ell + 1} - A c^{\ell} \right) &= 0, \quad\qquad i = 0, \dots, s - 1 .
        \end{align}
    \end{subequations}
    When \cref{eq:ODE} is linear, i.e., $y'(\timevar) = J y(\timevar) + g(t)$, the LTE depends only on the expressions \cref{Eq:ErrWSOPsi} and order conditions \cref{eq:linear_conditions}. Namely, if \cref{eq:linear_conditions} holds for \revision{$\ell = 0, \dots, p-1$} then the LTE for linear problems is \revision{of} order $p+1$. For instance, see \cite{skvortsov2003accuracy,rang2014analysis} \cite[p.~226]{hairer1996solving} for linear ODEs and \cite{ostermann1992runge} for linear PDEs. \Cref{eq:linear_conditions:WSO} was referred to as the \textit{weak stage order} conditions in \cite{rosales2024spatial, ketcheson2020dirk} and has been the subject of recent analysis \cite{biswas2022algebraic,biswas2023design,biswas2023explicit}.
\end{remark}

The structure of $\oc[\treevar][Z][\timevar]$ for the simplest ``non-bushy'' tree is highlighted next.
\begin{example}\label{Example:TallTree}
For $\treevar = \btree{[[[]]]}$,
\begin{align*}
    \oc[\treevar][Z][\timevar]
    &= (b^T \otimes I) (I - A \otimes Z)^{-1} (I \otimes g'(y(\timevar))) (I - A \otimes Z)^{-1} \left( \so{2} \otimes y''(\timevar) \right) \\
    &= (b^T \otimes I) \left( \sum_{i=0}^{s-1} A^i \otimes Q_i(Z) \right) (I \otimes g'(y(\timevar))) \left( \sum_{i=0}^{s-1} A^i \otimes Q_i(Z) \right) \left( \so{2} \otimes y''(\timevar) \right) \\
    &= \sum_{i=0}^{s-1} \sum_{j=0}^{s-1} (b^T A^i A^j \so{2}) Q_i(Z) g'(y(\timevar)) Q_j(Z) y''(\timevar) .
\end{align*}
For this tree the order conditions are $b^T A^i A^j \so{2} = 0$ for all $i, j = 0, \ldots, s-1$. The fact that the indices $i$ and $j$ are redundant for this tree is discussed in \cref{subsec:OC_reduction}.
\end{example}

To systematically characterize all the order conditions, we define a vector space $\Vs_{\treevar} \subseteq \R^s$ for each rooted tree as follows:
\begin{subequations}\label{eq:VSpace}         
    \newline 
    \noindent
    If $\treevar = [\treeroot^{\ell}]$ then
    \begin{equation}\label{eq:VSpaceRoot} 
        \Vs_{\treevar} \coloneqq 
        \textrm{span}\Big\{ A^j \so{\ell + 1} : \forall \; j = 0, 1, \ldots, s-1 \Big\}. 
    \end{equation}    
    \noindent If $\treevar = [\treeroot^{\ell} \, \treevar_1 \dots \treevar_k]$ then 
    \begin{equation}
    \label{eq:VSpaceGen}
        \Vs_{\treevar} \coloneqq \textrm{span} \Big\{ 
        A^{j+1} C^\ell \big( \beta_1 \times \cdots \times \beta_k \big) :
        \forall j = 0, \ldots, s-1, \; \forall \beta_i \in \Vs_{\treevar_i} \; (i = 1, \ldots, k) 
    \Big\}.
    \end{equation}
\end{subequations}
The spaces $\Vs_{\treevar}$ then provide a basis for expanding $\oc*[\treevar][Z][\timevar]$.
\begin{lemma}\label{Lem:VectorRepPsi}
    The vectors $\oc*[\treevar][Z][\timevar]$ defined in \cref{eq:VSpace} can be expressed as
    \begin{equation}\label{Eq:BasisPsi}
        \oc*[\treevar][Z][\timevar] = \sum_{j = 1}^{\dim \Vs_{\treevar}} \beta_j^{\treevar} \otimes u_j^{\treevar},
    \end{equation}
    where $\beta_{j}^{\treevar} \in \Vs_{\treevar}$ and $u_j^{\treevar} = u_j^{\treevar}(Z,t) \in \R^{\nvar}$ for $j=1, \ldots, \dim \Vs_{\treevar}$.
\end{lemma}
\begin{proof}
First observe that when $\treevar = [\treeroot^{\ell}]$, substituting \cref{Eq:MatExpansion} into \cref{eq:LTE_tree:Stage} yields
    \begin{equation}\label{Eq:BaseCase}
        \oc*[\treevar][Z][\timevar] =  \sum_{j = 0}^{s-1} A^{j} \so{\ell + 1} \otimes Q_j(Z) y^{(\ell+1)}(\timevar) .
    \end{equation}
    Since \cref{Eq:BaseCase} has the form \cref{Eq:BasisPsi}, the Lemma holds for $\treevar = [\treeroot^{\ell}]$.

    Next we proceed by strong induction on the number of vertices, $|\treevar|$, of $\treevar$. The base case of $\treevar = [\treeroot]$ corresponds to \cref{Eq:BaseCase} with $\ell = 1$ and is already established ($\treevar = \treeroot$ is trivial since $\oc*[\treeroot][Z][\timevar] = 0$). 

    Fix $m > 1$ and assume that \cref{Eq:BasisPsi} holds for all $|\treevar| \leq m$. Now consider any tree $\treevar$ for which $|\treevar| = m + 1$.  If $\treevar = [\treeroot^{m}]$ then we are done. If $\treevar = [\treeroot^{\ell} \treevar_1 \cdots \treevar_k]$ then $|\treevar_j| \leq m$, so that the induction hypothesis holds for each $\oc*[\treevar_j][Z][\timevar]$. 

    For each $\oc*[\treevar_j][Z][\timevar]$, substitute the expression \cref{Eq:BasisPsi} into $\G$ and use \cref{Eq:KroneckerIdentity} to obtain
    \begin{align}\nonumber
            & \quad \G(\timevar) \Big( 
                \sum_{j_1} \beta_{j_1}^{\treevar_1} \otimes u_{j_1}^{\treevar_1} , \dots, 
                \sum_{j_k} \beta_{j_k}^{\treevar_k} \otimes u_{j_k}^{\treevar_k} \Big) \\
            &=\sum_{j_1, \ldots, j_k} \G(\timevar) \Big( \beta_{j_1}^{\treevar_1} \otimes u_{j_1}^{\treevar_1} , \ldots, \beta_{j_k}^{\treevar_k} \otimes u_{j_k}^{\treevar_k} \Big) \nonumber \\ \label{Eq:GPsiSum}
            &= \sum_{j_1, \ldots, j_k} \Big(\beta_{j_1}^{\treevar_1}\times \cdots \times \beta_{j_k}^{\treevar_k}\Big) \otimes \gk(\timevar)\Big( u_{j_1}^{\treevar_1} , \ldots, u_{j_k}^{\treevar_k} \Big).
    \end{align}
    Here the summation variables $j_\nu$ ($\nu = 1, \ldots, k$) run from $1$ through $\dim \Vs_{\treevar_i}$. Substituting \cref{Eq:GPsiSum} and \cref{Eq:MatExpansion} into \cref{eq:LTE_tree:y} yields the result.
\end{proof}
The semilinear order conditions then arise as orthogonality conditions between the vector $b$ and expressions involving the spaces $\Vs_{\treevar}$. 
\begin{definition}[Semilinear Order Conditions]\label{def:semilinear_order} A Runge--Kutta method has semilinear order $\psl \geq 1$ if for all trees $|\treevar| \leq \psl$ the following algebraic conditions hold:
    \begin{subequations}\label{eq:OC}    
        \newline
        \noindent When $\treevar = [\treeroot^{\ell}]$, the conditions are
        \begin{equation}\label{eq:OC_Root} 
            \qo{\ell+1} = 0 \qquad \textrm{and} \qquad b^T \beta = 0 \qquad \forall \beta \in V_{\treevar}.
        \end{equation}    
        \noindent When $\treevar = [\treeroot^{\ell} \, \treevar_1 \dots \treevar_k]$, the conditions are
        \begin{equation}\label{eq:OC_GenTree} 
            b^T A^{j} C^{\ell} \Big( \beta_1 \times \cdots \times \beta_k \Big) = 0         
        \end{equation}
        for all $j = 0, \ldots, s-1$ and all sets of vectors $\beta_1 \in \Vs_{\treevar_1}$, $\ldots$, $\beta_k \in \Vs_{\treevar_{k}}$.
    \end{subequations}
\end{definition}

\begin{remark}
    Semilinear order is weaker than stage order; and a Runge--Kutta method has stage order $q \leq \psl$.
\end{remark}

Note that \cref{eq:OC} are polynomial equations defined in terms of $(A,b,c)$. As a result, \cref{def:semilinear_order} is a well-defined property for any Runge--Kutta method (regardless of whether the method is implicit, explicit, or satisfies AS- or ASI-stability). 

\begin{theorem}[Main Result for LTE]\label{thm:MainLTE} 
    Let $(A,b,c)$ be an AS- and ASI-stable Runge--Kutta method with semilinear order $\psl$. 
    Then the Runge--Kutta method applied to any initial value problem \cref{eq:ODE} satisfying \cref{assump:ODE}, has 
    \begin{equation}\label{Eq:LTE_MainResult}
        \oc[\treevar][Z][\timevar] = 0 \qquad \textrm{for all} \qquad |\treevar| \leq \psl.
    \end{equation}
    In particular, there are constants $D$, $\widetilde{h}$ depending only on $L$, $M$, and the method coefficients (but not on $J$ or $Z$) for which 
    \begin{equation}\label{Eq:LTE_ErrorEst}
        \norm{\dy[1]} \leq D \dt^{\psl + 1} \qquad \forall h \in [0, \widetilde{h}) .
    \end{equation}
\end{theorem}
\begin{proof}
Under the assumptions, the functions $\oc[\treevar][Z][\timevar]$ and $\oc*[\treevar][Z][\timevar]$ exist with $\oc*[\treevar][Z][\timevar]$ given by \cref{Eq:BasisPsi} for all $h \in [0, \widetilde{h})$ (with $\widetilde{h}$ from \cref{thm:exist_unique}). Note that \cref{rem:WSO} already established that \cref{eq:OC_Root} implies $\oc[[\treeroot^{\ell}]] = 0$.
For $\treevar \neq [\treeroot^{\ell}]$ the proof that $\oc = 0$ is identical to the last step in \cref{Lem:VectorRepPsi}. Substituting \cref{Eq:GPsiSum} and \cref{Eq:MatExpansion} into \cref{eq:LTE_tree:psi} yields the result.  Finally, substituting \cref{Eq:LTE_MainResult} into \cref{eq:LTE_tree_series:y} yields \cref{Eq:LTE_ErrorEst}.
\end{proof}

The semilinear order conditions, up to order five, are presented in \cref{tab:order_conditions}. For the purposes of clarity, the formulation in \cref{tab:order_conditions} uses indices $i_1, i_2, \ldots \in \{0, \ldots, s-1\}$ to define a set vectors that span $\Vs_{\treevar}$ as opposed to a minimal set of vectors that define a basis for $\Vs_{\treevar}$.
As a result, many conditions become linearly dependent.

\begin{table}[ht!]
    \centering
    \begin{tabular}{r|c|l|l}
        Label & Tree $\treevar$ & Order Condition $(\forall i_1, i_2, i_3, i_4 \in \{0, \dots, s - 1 \})$ & Implied By \\ \hline
        1a & \btree{[]} & $0 = 1 - b^T \one$ & $B(1)$ \\ \hline
        2a & \btree{[[]]} & $0 = \frac{1}{2} - b^T c = b^T A^{i_1} \left( \frac{c^{2}}{2} - A c \right)$ & $B(2), C(2)$ \\ \hline
        3a & \btree{[[][]]} & $0 = \frac{1}{6} - \frac{b^T c^2}{2} =  b^T A^{i_1} \left( \frac{c^{3}}{6} - \frac{A c^2}{2} \right)$ & $B(3), C(3)$ \\
        3b & \btree{[[[]]]} & $0 = b^T A^{i_1 + i_2} \left( \frac{c^{2}}{2} - A c \right)$ & 2a \\ \hline
        4a & \btree{[[][][]]} & $0 = \frac{1}{24} - \frac{b^T c^3}{6} =  b^T A^{i_1} \left( \frac{c^{4}}{24} - \frac{A c^3}{6} \right)$ & $B(4), C(4)$ \\
        4b & \btree{[[[]][]]} & $0 = b^T A^{i_1} C A^{i_2} \left( \frac{c^{2}}{2} - A c \right)$ & $C(2)$ \\
        4c & \btree{[[[][]]]} & $0 = b^T A^{i_1 + i_2} \left( \frac{c^{3}}{6} - \frac{A c^2}{2} \right)$ & 3a \\
        4d & \btree{[[[[]]]]} & $0 = b^T A^{i_1 + i_2 + i_3 + 1} \left( \frac{c^{2}}{2} - A c \right)$ & 2a \\ \hline
        5a & \btree{[[][][][]]} & $0 = \frac{1}{120} - \frac{b^T c^4}{24} =  b^T A^{i_1} \left( \frac{c^{5}}{120} - \frac{A c^4}{24} \right)$ & $B(5), C(5)$ \\
        5b & \btree{[[[]][][]]} & $0 = b^T A^{i_1} C^2 A^{i_2} \left( \frac{c^{2}}{2} - A c \right)$ & $C(2)$ \\
        5c & \btree{[[[]][[]]]} & $0 = b^T A^{i_1} \left( \left( A^{i_2} \left( \frac{c^{2}}{2} - A c \right) \right) \times \left( A^{i_3} \left( \frac{c^{2}}{2} - A c \right) \right) \right)$ & $C(2)$ \\
        5d & \btree{[[[][]][]]} & $0 = b^T A^{i_1} C A^{i_2} \left( \frac{c^{3}}{6} - \frac{A c^2}{2} \right)$ & $C(3)$ \\
        5e & \btree{[[[[]]][]]} & $0 = b^T A^{i_1} C A^{i_2 + i_3 + 1} \left( \frac{c^{2}}{2} - A c \right)$ & 4b \\
        5f & \btree{[[[][][]]]} & $0 = b^T A^{i_1 + i_2} \left( \frac{c^{4}}{24} - \frac{A c^3}{6} \right)$ & 4a \\
        5g & \btree{[[[[]][]]]} & $0 = b^T A^{i_1 + i_2 + 1} C A^{i_3} \left( \frac{c^{2}}{2} - A c \right)$ & 4b \\
        5h & \btree{[[[[][]]]]} & $0 = b^T A^{i_1 + i_2 + i_3 + 1} \left( \frac{c^{3}}{6} - \frac{A c^2}{2} \right)$ & 3a \\
        5i & \btree{[[[[[]]]]]} & $0 = b^T A^{i_1 + i_2 + i_3 + i_4 + 2} \left( \frac{c^{2}}{2} - A c \right)$ & 2a       
    \end{tabular}
    \caption{Semilinear order conditions associated with trees up to order five. Of the 17 trees shown, eight are redundant, while the remaining nine can be satisfied with the simplifying assumptions \cref{eq:simplifying_assumptions}. In order condition 5c, the $\times$ denotes an element-wise vector product.}
    \label{tab:order_conditions}
\end{table}

\subsection{Reduction of the semilinear order conditions} \label{subsec:OC_reduction}
Not every tree, corresponding to a row in \cref{tab:order_conditions}, yields an independent order condition. Some order conditions are implied by lower order conditions. For instance, \cref{tab:order_conditions}(3b) follows from \cref{tab:order_conditions}(2a): the Cayley--Hamilton theorem implies $A^{i + j}$ is a linear combination of $A^k$ for $k = 0, \ldots, s-1$. Hence, the order conditions in \cref{tab:order_conditions}(2a), i.e., $b^T A^k \so{2} = 0$ ($k = 0, \ldots, s-1$), are sufficient to ensure \cref{tab:order_conditions}(3b) hold.

More generally, some trees (those with certain internal vertices) can be removed from the set of semilinear order conditions.

\begin{lemma}\label{lem:redundtrees} Suppose $\treevar$ has a vertex $v$ with exactly one child, and the child is not a leaf.  Let $\tilde{\treevar}$ be the tree obtained by suppressing $v$ from $\treevar$.  Then $V_{\treevar} \subseteq V_{\tilde{\treevar}}$. If the semilinear order conditions for $\tilde{\treevar}$ hold, then so do the conditions for $\treevar$. 
\end{lemma}

Trees that do not satisfy the conditions in \cref{lem:redundtrees} are \emph{semi-lone-child-avoiding} (see \cite[A331934]{oeis}).  For example, the following two trees have a single vertex that can be suppressed to give a semi-lone-child-avoiding tree:
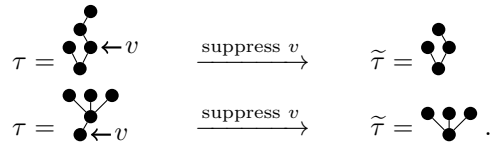
\begin{equation*}
    \tau =
    \begin{forest}
    btree [[,name=test,[[]]][]];
    \draw[<-,thick] ([xshift=0.15em]test.east) -- ++(0.75em,0) node[anchor=west,inner sep=0.1em] {$v$};
    \end{forest}
    \ \xrightarrow{\text{suppress } v} \ \
    \widetilde{\tau} = \btree{[[[]][]]}
\qquad\text{and}\qquad
    \tau =
    \begin{forest}
    btree [,name=test,[[][][]]];
    \draw[<-,thick] ([xshift=0.15em]test.east) -- ++(0.75em,0) node[anchor=west,inner sep=0.1em] {$v$};
    \end{forest}
    \ \ \xrightarrow{\text{suppress } v} \ \
    \widetilde{\tau} = \btree{[[][][]]}\;.
\end{equation*}

The proof of \cref{lem:redundtrees} makes use of two basic facts of the vectors spaces $\Vs_{\treevar}$. First, for any rooted tree $\treevar$, the space $\Vs_{\treevar}$ is $A$-invariant, i.e., $A u \in \Vs_{\treevar}$ for all $u \in \Vs_{\treevar}$. 
Secondly, if $\treevar = [\treeroot^{\ell} \treevar_1 \treevar_2 \cdots \treevar_k]$ and $\tilde{\treevar} = [\treeroot^{\ell} \tilde{\treevar}_1 \tilde{\treevar}_2 \cdots \tilde{\treevar}_k]$ are two trees satisfying $\Vs_{\treevar_j} \subseteq \Vs_{\tilde{\treevar}_j}$ for all $j = 1, \ldots, k$, then $\Vs_{\treevar} \subseteq \Vs_{\tilde{\treevar}}$.  

\begin{proof}[Proof of {\cref{lem:redundtrees}}]
    Under the assumptions in the theorem we have $\treevar \neq [\treeroot^{\ell}]$.
    First, suppose that $v$ is the root, in which case $\treevar = [\tilde{\treevar}]$. Then using the definition of \cref{eq:VSpaceGen} for both $\Vs_{\treevar}$ and $\Vs_{\tilde{\treevar}}$ implies (using the $A$-invariance of $\Vs_{\tilde{\treevar}}$):
    \begin{equation}\label{Eq:Inclusion}
        \Vs_{\treevar} = \textrm{span}\{ A^{j+1}\beta :  \beta \in \Vs_{\tilde{\treevar}}, \forall j = 0, \ldots, s-1 \} \subseteq \Vs_{\tilde{\treevar}}.
    \end{equation}
    If $v$ is not the root, let $\treevar_{v}$ be the subtree with root $v$ and set $\treevar_v = [\tilde{\treevar}_v]$. Then applying the same argument in \cref{Eq:Inclusion} to $\treevar_v$ yields $\Vs_{\treevar_v} \subseteq \Vs_{\tilde{\treevar}_v}$. 
    Let $p$ be the parent of $v$. The subtree of $\treevar$ with root $p$ has the form $\treevar_{p} = [ \treeroot^{\ell} \treevar_1 \cdots \treevar_k \treevar_v]$. The same subtree of $\tilde{\treevar}$ has the form $\tilde{\treevar}_{p} = [ \treeroot^{\ell} \treevar_1 \cdots \treevar_k \tilde{\treevar}_v]$. Hence, by the inclusion property of $\Vs$, $\Vs_{\treevar_p} \subseteq \Vs_{\tilde{\treevar}_{p}}$ since $\Vs_{\treevar_v} \subseteq \Vs_{\tilde{\treevar}_v}$. Applying this argument recursively by ascending the tree $\treevar$ yields $\Vs_{\treevar} \subseteq \Vs_{\tilde{\treevar}}$.
\end{proof}

\begin{corollary}[Reduction of semilinear order conditions]\label{cor:reduct_oc} A Runge--Kutta scheme has semilinear order $\psl$ if the conditions \cref{eq:OC} hold for all semi-lone-child-avoiding trees $\treevar$ satisfying $|\treevar| \leq \psl$.   
\end{corollary}

\section{Global error estimates}\label{sec:stability_convergence}
Here we show how the LTEs accumulate in semilinear problems to yield a global error of order $\psl$. The error bounds hold uniformly with respect to stiffness. In the spirit of \cite[Section~2.3]{hundsdorfer2003numerical} and \cite{ostermann1992runge}, when an additional property of the Runge--Kutta method holds, the global error admits an extra order of convergence, i.e., \emph{superconvergence} of order $\psl + 1$. 
The superconvergence result hinges on a telescoping series based on the next lemma which describes the evolution of two neighboring Runge--Kutta solutions. The lemma is an extension of ``C-stability'' \cite[Definition~2.13]{dekker1984stability} without the use of norms and is proven in \cref{app:C-stability-lemma}.

\begin{lemma} \label{lemma:C-stability}
    Suppose that \cref{assump:ODE} holds, and that an AS- and ASI-stable Runge--Kutta method is applied to \cref{eq:ODE}.  Then there exists $\overline{h} > 0$ (depending on $L$ and method coefficients but not on $J$) such that any two numerical solutions of \cref{eq:RK} to \cref{eq:ODE} satisfy
    \begin{equation} \label{eq:C-stability}
        y_{n+1} - \widetilde{y}_{n+1} = (R(Z) + \dt \Lambda_n) (y_n - \widetilde{y}_n),
        \quad
        \forall h \in [0, \overline{h}).
    \end{equation}
    The norm of matrix $\Lambda_n$ can be bounded in terms of only $L$ and the method coefficients.
\end{lemma}

The next theorem establishes the global error. We use the following condition on the Runge--Kutta stability function,
\begin{equation}\label{Eq:Rcondition}
    \lim_{z \rightarrow \infty} R(z) \neq 1\,, \qquad \textrm{and} \qquad z^{-1}(1-R(z)) \; \textrm{ has no zeros in } \; \C^{-} ,    
\end{equation}
to show that when constant time steps are used, an additional order is obtained. 

\begin{theorem}[Global Runge--Kutta error]\label{thm:convergence}
    Suppose a Runge--Kutta method is A-, AS-, and ASI-stable, has classical order $p$, and semilinear order $\psl \geq 1$. Under \cref{assump:ODE} (with $r \geq p, \psl$) there exists $\overline{\dt}, D > 0$ depending on $L$, $M$, $\tf-\ts$, and the Runge--Kutta coefficients, but not on $J$, such that
    \begin{equation} \label{eq:global_error}
        \begin{split}
            \norm{\dy} &\leq D \dt^{q},
            \qquad
            \text{with} \quad
            \dt \in [0, \overline{\dt}), \quad
            n \dt \leq \tf-\ts, \quad y_0 = y(\ts), \\
            q &= \begin{cases}
                \psl + 1, & \text{if $p = \psl + 1$ and $R(z)$ satisfies \cref{Eq:Rcondition}} \\
                \psl, & \text{otherwise}\;.
            \end{cases}
        \end{split}
    \end{equation}
\end{theorem}
\begin{proof}
    Here we use $D_i$ to denote a positive constant depending only on $L$, $M$, $\tf-\ts$, and the method coefficients.
    Take $\dt < \overline{\dt}$ (where $\overline{h}$ is defined in \cref{lemma:C-stability}) and let $y_n$ denote the Runge--Kutta solution initialized to $y_0 = y(\ts)$.  For each $n$, let $\widetilde{y}_{n}$ be one step of the Runge--Kutta method initialized with the exact solution $y(\timevar_{n-1})$. Hence, $y(\timevar_n) - \widetilde{y}_n$ is the LTE \revision{at $t_n$}.
    Using \cref{lemma:C-stability}, the Runge--Kutta error is
    \begin{equation}
    \label{Eq:ErrRecursion}
        \dy
        = (\widetilde{y}_n - y_n) + (y(\timevar_n) - \widetilde{y}_n)
        = (R(Z) + \dt \Lambda_{n-1}) \dy[n-1] + (y(\timevar_n) - \widetilde{y}_n).
    \end{equation}
    The ``standard'' global error bound then follows from using A-stability and \cite[Theorem~4]{hairer1982stability} to bound the first term in \cref{Eq:ErrRecursion}, and \cref{thm:LTE} to bound the second, whence
    \begin{equation*}
        \norm{\dy} \leq (1 + h D_1) \norm{\dy[n-1]} + D_2 \dt^{\psl+1}.
    \end{equation*}
    Thus, $q = \psl$ follows from a geometric series bound. 
        
    The case of $q = \psl + 1$ relies on a telescoping series for the error. First, separate out the leading order of the LTE in \cref{Eq:ErrRecursion} as 
    \begin{equation*}
        y(\timevar_n) - \widetilde{y}_n = \big(I - R(Z)\big) \xi_{n-1} + \eta_{n-1},
    \end{equation*}
    where    
    \begin{equation*}
        \xi_{n-1} = \big(I - R(Z)\big)^{-1} \sum_{\treevar \in T_{\psl + 1}} \, \zeta_{\treevar} \, \oc[\treevar][Z][t_{n-1}] \, \dt^{\psl + 1} ,
    \end{equation*}
    and $\norm{\eta_{n-1}} \leq D_3 \dt^{\psl + 2}$  where $D_3$ is independent of $Z$.  In the spirit of \cite[Lemma~2.3]{hundsdorfer2003numerical}, the shifted error $\epsilon_n = \dy - \xi_n$ satisfies the forced linear recurrence
    \begin{equation}\label{Eq:ShiftedLinRec}
        \epsilon_n = \big(R(Z) + \dt \Lambda_{n-1}\big) \epsilon_{n-1} + (\xi_{n-1} - \xi_n) + \eta_{n-1} + \dt \Lambda_{n-1} \xi_{n-1},
        \qquad
        \epsilon_0 = -\xi_0.
    \end{equation}
    We now claim that there exists $D_4 > 0$ such that $\forall \, n \in \N$ for which $n \dt \leq \tf-\ts$, and $Z$ satisfying $\mu(Z) \leq 0$, the following estimates hold:
    \begin{equation}\label{Eq:EstimateXi}
        \norm{\xi_n} \leq D_4 \dt^{\psl + 1} \qquad \textrm{and} \qquad 
        \norm{\xi_n - \xi_{n-1}} \leq D_4 \dt^{\psl + 2}.
    \end{equation}
    To establish \cref{Eq:EstimateXi}, introduce the function $F_{\tau}(Z, t) \coloneqq (I - R(Z))^{-1} \oc[\treevar][Z][t]$,
    where $\psi$ is defined in \cref{thm:LTE_trees}. When $|\tau| = p \leq r$, the function $F_{\tau}(z, t)$ is a rational function in the variable $z \in \C$, whose coefficients are continuously differentiable functions of $t$ (i.e., the coefficients depend on the derivatives of $g$, $y$ up to orders $|\tau|-1$ and $|\tau|$, respectively). 

    We first show that, for fixed $\tau$, the function $F_{\tau}(z, t)$ of $z$ and $t$, as well as its time derivative, are bounded in $\C^{-} \times [\ts, \tf]$. For $|\tau| = p$, the function $F_{\tau}(z, t)$ has no poles in $\C^{-} \times [\ts, \tf]$ since $\ldots$
    \begin{itemize}
        \item by \cref{Eq:Rcondition}, the rational function $(1 - R(z))^{-1}$ has a single simple pole at $z=0$ on $\C^{-}$;
        \item $\oc[\treevar][z][t]$ is a rational function of $z$, and bounded on $z \in \C^{-}$ (by AS-, ASI- stability);
        \item for $|\tau| = p$, the function $\oc[\treevar][0][t]$ is a classical $p = \psl+1$ order condition (see \cref{rem:classical_special_case}), which by assumption is zero.
    \end{itemize}
    We can conclude that the pole at $z = 0$ in $F_{\tau}(z, t)$ is removable. 
    
    Next, introduce the conformal mapping $\rho(w) \coloneqq (w-1)/(w+1)$ which maps $\rho : \mathcal{D} \rightarrow \C^{-}$ where $\mathcal{D} \coloneqq \{ z \in \C : |z| \leq 1\}$ is the closed unit disk. Then for $|\tau| = p$, both $F_{\tau}(\rho(w), t)$ and $\partial F_{\tau}(\rho(w), t)/\partial t$ are rational functions of $w$, with coefficients that are continuous functions of $t$, and poles outside $\mathcal{D}$ (by the quotient rule, differentiation does not change the location of a pole). Hence, both $|F|$ and $|\partial F/\partial t|$ are continuous functions on the compact set $\mathcal{D} \times [\ts, \tf]$, and thus bounded. 

    The scalar version of \cref{Thm:Nevanlinna} (cf.~\cite[Theorem~4]{hairer1982stability}) implies that for $|\tau| = p$, $Z$ satisfying $\mu(Z) \leq 0$ and $n \dt \leq \tf-\ts$, the following estimates hold:
    \begin{align}
        \label{eq:leading_error_function1}        
        \norm{F_{\tau}(Z, t_n)} &\leq \sup_{(z, t) \in \C^{-} \times [\ts,\tf]} |F_{\tau}(z, t)| < \infty,
        \\
        \nonumber
        \norm{F_{\tau}(Z, t_n) - F_{\tau}(Z, t_{n-1})} &\leq
        \sup_{(z, t) \in \C^{-} \times [\ts,\tf]} \abs{F_{\tau}(z, t_{n}) - F_{\tau}(z, t_{n-1})} \\
        &\leq D_5 \abs{t_{n} - t_{n-1}} = D_5 \dt .
        \label{eq:leading_error_function2}
    \end{align}
    Here $D_5$ is the Lipschitz constant provided by the bound on the time derivative of $F$. Applying \cref{eq:leading_error_function1} and \cref{eq:leading_error_function2} with the triangle inequality to $\xi_n$ yields \cref{Eq:EstimateXi}. 

    Repeating the standard geometric series bound on \cref{Eq:ShiftedLinRec}, together with \cref{Eq:EstimateXi}, we can conclude that $\norm{\epsilon_n} \leq D_6 \dt^{\psl + 1}$. Since $\dy = \epsilon_n + \xi_n$, it follows that \cref{eq:global_error} holds.
\end{proof}

\section{Conclusions} \label{sec:conclusion}
Stiff convergence analysis is a key complement to the classical error analysis in the asymptotic limit of the time step being small relative to the two-sided Lipschitz constant of the ODE's right-hand side. This work has established such stiff convergence theory for Runge--Kutta methods applied to semilinear ODEs in which the linear term can be arbitrarily stiff. The theory provides rigorous error estimates outside of the classical asymptotic regime, when implicit methods are of interest. Specifically, the established B-convergence results hold uniformly with respect to stiffness.


The approach employed herein to derive the semilinear order conditions adapts a unique recursion originally proposed by Albrecht \cite{albrecht1987new,albrecht1996runge}.
Up to order three in the LTE, the semilinear order conditions coincide with conditions already known for linear problems \cite{biswas2022algebraic}. This structural insight in particular rationalizes why existing methods with high weak stage order manage to mitigate order reduction on problems outside what previous theory could predict. Starting at fourth order terms in the LTE, new order conditions arise, i.e., semilinear order strictly goes beyond weak stage order. All of the semilinear order conditions have been established to be in one-to-one correspondence with rooted trees.

The theoretical framework established herein will be leveraged in a companion paper devoted to the derivation of various novel DIRK methods of up to order 5, as well as the demonstration of the successful mitigation of order reduction in relevant semilinear test problems, for which existing comparable methods fail to exhibit their full order of convergence.



\backmatter

\bmhead{Acknowledgements}
The authors would like to thank David I. Ketcheson for many helpful discussions.
DS would like to gratefully acknowledge the Vietnam Institute for Advanced Study in Mathematics (VIASM) for hosting a visit during June 2025.

\section*{Declarations}
This work was performed under the auspices of the U.S.\ Department of Energy by Lawrence Livermore National Laboratory under Contract DE-AC52-07NA27344.
LLNL-JRNL-2005958.
Roberts was supported by the Fernbach Fellowship through the LLNL-LDRD Program under Project No. 23-ERD-048.
This material is based upon work supported by the National Science Foundation under Grant No.~DMS--2309728 (Seibold) and DMS--2309727 (Shirokoff). Any opinions, findings, and conclusions or recommendations expressed in this material are those of the authors and do not necessarily reflect the views of the National Science Foundation.

\begin{appendices}

\section{Proof of Theorem~\ref{thm:LTE_trees}} \label{app:abstractrec}
We establish \cref{thm:LTE_trees} by viewing \cref{eq:LTE_coeffs} as an abstract recursion relation. Let $\mathbb{V}$ be a finite dimensional vector space over $\R$. Given a sequence of vectors $\ac[i]$ $\revision{(i = 1, 2, \dots)}$ in $\mathbb{V}$,  families of numbers $\scalvar_{\ell, k} \in \R$ and symmetric $k$-linear maps $\bimap_{\ell, k} : \revision{\mathbb{V}^{k}} \rightarrow \mathbb{V}$ indexed by $\ell \geq 0$ and $k \geq 1$, define the recursion relation:
\begin{align}
    \nonumber
    \vy[1] &\coloneqq 0 \\
    \label{eq:abstractrec}
    \vy[i+1] &\coloneqq \ac[i] + \sum_{\ell = 0}^{i-1} \ \sum_{k=1}^{i-\ell} \
    \sum_{m_1 + \ldots + m_k = i - \ell}
    \scalvar_{\ell, k} \bimap_{\ell, k}(\vy[m_1], \ldots, \vy[m_k])    \qquad \revision{(i = 1, 2, \dots)}.
\end{align}
Note that both the linear recursion for $\dY[0][i]$ in \cref{eq:LTE_coeffs} can be recast in the form \cref{eq:abstractrec} with suitably chosen variables, maps and a vector space $\mathbb{V}$.

\begin{lemma}\label{lem:abstractrec}
    The recurrence \cref{eq:abstractrec} has a solution of the form 
    \begin{equation}\label{eq:abstractsol}
        \vy[i]  = \sum_{\treevar \in T_{i}} \comcoef(\treevar) \, \varphi(\treevar),
    \end{equation}
    where $\varphi : T \rightarrow \mathbb{V}$ is given by 
    \begin{equation}\label{Eq:PhiFormula}
        \varphi(\treevar) =  \left\{\begin{array}{lll}
            0 & \quad \treevar = \treeroot \\
            \ac[\ell] & \quad \treevar = [\treeroot^{\ell}], & (\ell \geq 1) \\ 
            \bimap_{\ell, k}(\varphi(\treevar_1), \ldots, \varphi(\treevar_k) ) & \quad \treevar = [\treeroot^{\ell} \, \treevar_1 \dots \treevar_k], \, & (\ell \geq 0, k \geq 1)
        \end{array} \right. .
    \end{equation}
    Here $\comcoef : T \rightarrow  \R$ is a combinatorial factor defined to be
    \begin{equation}\label{eq:SigmaFormula}
        \comcoef([\treeroot^{\ell}]) = 1, \qquad 
        \comcoef([\treeroot^{\ell} \treevar_1 \ldots \treevar_k]) = \frac{\scalvar_{\ell, k} \, k!}{\mu_1! \, \mu_2! \cdots \mu_\sigma!}\comcoef(\treevar_1) \comcoef(\treevar_2) \cdots \comcoef(\treevar_k),
        \quad \ell \geq 0, k \geq 1
    \end{equation}
    where $\mu_1, \ldots, \mu_\sigma$ are the multiplicities of the $\sigma$ distinct trees in the set $\{\treevar_1, \ldots, \treevar_k\}$. 
\end{lemma}

\begin{remark}
    If $\ell = 0$ and $\scalvar_{\ell, k} = 1$, then $\comcoef$ counts the permutations of $\treevar_1, \treevar_2, \ldots, \treevar_k$ (cf.~\cite[Chapter~III.1.3]{hairerlubichwanner2006}).
\end{remark}

\begin{proof}
The proof of \cref{lem:abstractrec} follows via induction on $i$. For $i = 1$, we have $\vy[1] = \varphi(\treeroot) = 0$.  Next, assume that the formulas \cref{eq:abstractsol,Eq:PhiFormula,eq:SigmaFormula} hold for $1, \ldots, i$. We then show the result holds for $\vy[i+1]$.
%

Substituting the ansatz \cref{eq:abstractsol} into \cref{eq:abstractrec}, we have
\begin{align}\nonumber 
    \vy[i+1] &= \ac[i] + \sum_{\ell=0}^{i-1} \sum_{k=1}^{i-\ell} \
    \sum_{\substack{ m_1 + \ldots + \\ m_k = i - \ell } }
    \scalvar_{\ell, k}  \bimap_{\ell, k}\mleft( \sum_{\treevar_1 \in T_{m_1}} \comcoef(\treevar_1) \varphi(\treevar_1)\,, \ldots, \sum_{\treevar_k \in T_{m_k}}\comcoef(\treevar_k) \varphi(\treevar_k) \mright) \nonumber \\ \label{eq:sum1}
    &= \ac[i] + \sum_{\ell =0}^{i-1} \sum_{\substack{ \treevar \in T_{i - \ell+1} \\ 
    \treevar = [\treevar_1 \treevar_2 \cdots \treevar_k]} } \frac{ \scalvar_{\ell, k}  \, k!}{\mu_1! \ldots \mu_{\sigma}!} \ \comcoef(\treevar_1) \cdots \comcoef(\treevar_k) \, \bimap_{\ell, k}\big( \varphi(\treevar_1), \ldots, \varphi(\treevar_k) \big),
\end{align}
where the $\mu_1, \ldots, \mu_\sigma$ are defined as in \cref{eq:SigmaFormula}.  The first line of \cref{eq:sum1} sums over all trees $\treevar_1, \treevar_2, \ldots, \treevar_k$ (where $k$ can be arbitrary) so long as the vertices sum to $i - \ell$.  Since $\bimap_{\ell, k}$ is a symmetric multilinear function, the second line follows by summing over each tree $[\treevar_1 \treevar_2 \ldots \treevar_k]$ once.  The combinatorial factor $k!/(\mu_1!\cdots \mu_{\sigma}!)$ counts the permutations of $\treevar_1, \ldots, \treevar_k$ that yield the same tree $[\treevar_1 \treevar_2 \ldots \treevar_k]$.

Next, note that $\varphi(\treeroot) = 0$, so that any term in the summation in \cref{eq:sum1} with $\treevar_j = \treeroot$ is zero.  Therefore, one may assume, without loss of generality, that $\treevar = [\treevar_1 \treevar_2 \ldots \treevar_k]$ in the summation in \cref{eq:sum1} is in standard form with no power of $\treeroot$.  The summation over $\ell$ can then be represented by adjoining $\treeroot^{\ell}$ to the root of $[\treevar_1 \treevar_2 \ldots \treevar_k]$, which yields
\begin{equation}\label{eq:sum2}
    \vy[i+1] = \underbrace{\ac[i]}_{\varphi([\treeroot^i])} + \sum_{\substack{ \treevar \in T_{i + 1} \\ 
    \treevar = [\treeroot^{\ell} \treevar_1 \treevar_2 \cdots \treevar_k]} } \comcoef(\treevar) \, \underbrace{\bimap_{\ell, k}\big( \varphi(\treevar_1), \ldots, \varphi(\treevar_k) \big)}_{\varphi(\treevar)} .
\end{equation}
The last term in \cref{eq:sum2} sums over all trees in $T_{i+1}$ except  $[\treeroot^i]$ since no $\treevar_j = \treeroot$ for \revision{$j = 1, \dots, k$}.  The term $\ac[i]$ in \cref{eq:sum2} then adds the missing lone tree to the sum (since $\comcoef([\treeroot^{i}]) = 1$), and the result \cref{Eq:PhiFormula} holds for $i+1$.
\end{proof}

\begin{proof}[Proof of \cref{thm:LTE_trees}] The recursion relation for $\dY[0][i]$ in \cref{eq:LTE_coeffs} can be recast in the form \cref{eq:abstractrec} by taking $\mathbb{V} = \R^{s} \otimes \R^{\nvar}$, $\vy[i] = \dY[0][i]$, $\varphi(\treevar) = \oc*[\treevar][Z][t]$, and
    \begin{align*}
        \scalvar_{\ell, k} &= (-1)^{k+1}\frac{1}{k! \ell!},
        \qquad
        \ac[i] = (I - A \otimes Z)^{-1} \left( \so{i+1} \otimes y^{(i+1)}(\timevar) \right), \\
        \bimap_{\ell, k} &= 
            (I - A \otimes Z)^{-1} \big((A C^{\ell}) \otimes I\big) \, \G(\timevar).
    \end{align*}
    \Cref{thm:LTE_trees} for $\dY[0][i]$ then follows by applying \cref{lem:abstractrec}.
    When we substitute the proven value of $\dY[0][i]$ into \cref{eq:LTE_coeffs}, $\dy[1][i]$ can be written in the form of \cref{eq:sum1} with $\varphi(\treevar) = \oc*[\treevar][Z][t]$ and
    \begin{align*}
        \ac[i] &= \big(b^T \otimes Z\big) \big(I - A \otimes Z\big)^{-1} \left( \so{i+1} \otimes y^{(i+1)}(\timevar) \right) +\qo{i+1} y^{(i+1)}(\timevar), \\
        \bimap_{\ell, k} &= 
        \big(b^T \otimes I\big) \big(I - A \otimes Z\big)^{-1} \big(C^{\ell} \otimes I\big) \, \G(\timevar).
    \end{align*}
    \Cref{eq:LTE_tree_series:y} follows from the same simplification to sums over trees from the proof of \cref{lem:abstractrec}.
\end{proof}

\section{Proof of Lemma~\ref{lemma:C-stability}}\label{app:C-stability-lemma}
Take $\dt < \widetilde{\dt}$ (as defined in \cref{thm:exist_unique}) so that the stage equations for $Y_n$ are solvable.
Introduce $\Delta \widetilde{y}_n = y_n - \widetilde{y}_n$ and $\Delta \widetilde{Y}_n = Y_n - \widetilde{Y}_n$.
Following the same analysis as \cref{eq:LTE}, the difference in solutions yields (for $\dt < \widetilde{\dt}$) 
\begin{subequations}
    \begin{align}
        \label{eq:diff:stages}
        \Delta \widetilde{Y}_n &= (I - A \otimes Z)^{-1}(\one \otimes \Delta \widetilde{y}_n + \dt (A \otimes I)(\tp{g}(Y_n) - \tp{g}(\widetilde{Y}_n))), \\ \label{eq:diff:step}
        \Delta \widetilde{y}_{n+1} &= R(Z) \, \Delta \widetilde{y}_n + \dt (b^T \otimes I) (I - A \otimes Z)^{-1}(\tp{g}(Y_n) - \tp{g}(\widetilde{Y}_n)).
    \end{align}
\end{subequations}
Here AS- and ASI-stability ensure the existence and boundedness of the matrix inverses and $R(Z)$.
Focusing on the difference in $g$ values, we have
\begin{equation} \label{eq:MVT}
    \tp{g}(Y_n) - \tp{g}(\widetilde{Y}_n)
    = \begin{bmatrix}
        g(Y_{n,1}) - g(\widetilde{Y}_{n,1}) \\
        \vdots \\
        g(Y_{n,s}) - g(\widetilde{Y}_{n,s})
    \end{bmatrix}
    = \begin{bmatrix}
        \overline{G}_{n,1} \Delta \widetilde{Y}_{n,1} \\
        \vdots \\
        \overline{G}_{n,s} \Delta \widetilde{Y}_{n,s}
    \end{bmatrix}
    = \overline{G}_n \Delta \widetilde{Y}_{n},
\end{equation}
where $\overline{G}_{n,i}$ is the mean value of the Jacobian matrix:
\begin{equation*}
    \overline{G}_{n,i} \coloneqq \int_0^1 g'\mleft(\theta Y_{n,i} + (1 - \theta) \widetilde{Y}_{n,i} \mright) \, {\rm d}\theta,
    \qquad
    \overline{G}_n = \blkdiag(\overline{G}_{n,1},\dots,\overline{G}_{n,s}).
\end{equation*}
%
Substituting the right-hand side of \cref{eq:MVT} into \cref{eq:diff:stages} and solving for $\Delta \widetilde{Y}_n$ yields
\begin{equation*}
    \Delta \widetilde{Y}_n        
    = \left(I - h (I - A \otimes Z)^{-1} (A \otimes I) \overline{G}_n \right)^{-1} (I - A \otimes Z)^{-1} ( \one \otimes \Delta \widetilde{y}_n ), 
\end{equation*}
where the matrix $I - h (I - A \otimes Z)^{-1} (A \otimes I) \overline{G}_n$ is guaranteed to be invertible if
\begin{equation}\label{eq:invert_cond}
    h \norm{(I - A \otimes Z)^{-1} (A \otimes I) \overline{G}_n} < 1.
\end{equation}
A simple sufficient condition for \cref{eq:invert_cond} to hold is as follows. Let $B < \infty$ be an upper bound (implied by \cref{lem:AS_ASI_stability_bounds}) satisfying $\norm{(I - A \otimes Z)^{-1}} \leq B$.
Next note that \cref{assump:ODE} implies $\norm{\overline{G}_{n,i}} \leq L$  ($i = 1, \ldots, s$), which implies that $\norm{\overline{G}_n} \leq L$.
Thus,
\begin{equation*}
    \norm{(I - A \otimes Z)^{-1} (A \otimes I) \overline{G}_n}
    \leq L B \norm{A}. 
\end{equation*}
Setting $\overline{\dt} = \min \{ \widetilde{\dt}, \frac{1}{2} (L B \norm{A})^{-1} \}$ implies that \cref{eq:invert_cond} holds for $\dt < \overline{\dt}$. The extra factor of $\frac{1}{2}$ is chosen somewhat arbitrarily so that $\dt L B \norm{A}$ is bounded strictly away from $1$.
Lastly, substituting \cref{eq:diff:stages} and \cref{eq:MVT} into \cref{eq:diff:step} yields
\begin{equation*}
    \Delta \widetilde{y}_{n+1} = R(Z) \Delta \widetilde{y}_n + \dt \Lambda_n \, \Delta \widetilde{y}_n,
\end{equation*}
where $\Lambda_n$ is the matrix defined as
%
\begin{equation*}
    \Lambda_n \coloneqq (b^T \otimes I) (I - A \otimes Z)^{-1} \overline{G}_n
    (I - \dt (I - A \otimes Z)^{-1} (A \otimes I) \overline{G}_n)^{-1} (I - A \otimes Z)^{-1} (\one \otimes I).
\end{equation*}
Using that fact, statement \cref{eq:C-stability} follows because
\begin{equation*}
\begin{split}
    \norm{\Lambda_n}
    &\leq \norm{b} \cdot \norm{(I - A \otimes Z)^{-1}}^2 \cdot \norm{\overline{G}_n} \cdot \norm{(I - \dt (I - A \otimes Z)^{-1} (A \otimes I) \overline{G}_n)^{-1}} \cdot \norm{\one} \\
    &\leq \frac{\norm{b} B^2 L \sqrt{s}}{1 - \dt \norm{(I - A \otimes Z)^{-1} (A \otimes I) \overline{G}_n}} \\
    &\leq 2 \norm{b} B^2 L \sqrt{s}\;.
\end{split}
\end{equation*}
%

\end{appendices}

\bibliography{bib}


\begin{thebibliography}{39}
\ifx \bisbn   \undefined \def \bisbn  #1{ISBN #1}\fi
\ifx \binits  \undefined \def \binits#1{#1}\fi
\ifx \bauthor  \undefined \def \bauthor#1{#1}\fi
\ifx \batitle  \undefined \def \batitle#1{#1}\fi
\ifx \bjtitle  \undefined \def \bjtitle#1{#1}\fi
\ifx \bvolume  \undefined \def \bvolume#1{\textbf{#1}}\fi
\ifx \byear  \undefined \def \byear#1{#1}\fi
\ifx \bissue  \undefined \def \bissue#1{#1}\fi
\ifx \bfpage  \undefined \def \bfpage#1{#1}\fi
\ifx \blpage  \undefined \def \blpage #1{#1}\fi
\ifx \burl  \undefined \def \burl#1{\textsf{#1}}\fi
\ifx \doiurl  \undefined \def \doiurl#1{\url{https://doi.org/#1}}\fi
\ifx \betal  \undefined \def \betal{\textit{et al.}}\fi
\ifx \binstitute  \undefined \def \binstitute#1{#1}\fi
\ifx \binstitutionaled  \undefined \def \binstitutionaled#1{#1}\fi
\ifx \bctitle  \undefined \def \bctitle#1{#1}\fi
\ifx \beditor  \undefined \def \beditor#1{#1}\fi
\ifx \bpublisher  \undefined \def \bpublisher#1{#1}\fi
\ifx \bbtitle  \undefined \def \bbtitle#1{#1}\fi
\ifx \bedition  \undefined \def \bedition#1{#1}\fi
\ifx \bseriesno  \undefined \def \bseriesno#1{#1}\fi
\ifx \blocation  \undefined \def \blocation#1{#1}\fi
\ifx \bsertitle  \undefined \def \bsertitle#1{#1}\fi
\ifx \bsnm \undefined \def \bsnm#1{#1}\fi
\ifx \bsuffix \undefined \def \bsuffix#1{#1}\fi
\ifx \bparticle \undefined \def \bparticle#1{#1}\fi
\ifx \barticle \undefined \def \barticle#1{#1}\fi
\bibcommenthead
\ifx \bconfdate \undefined \def \bconfdate #1{#1}\fi
\ifx \botherref \undefined \def \botherref #1{#1}\fi
\ifx \url \undefined \def \url#1{\textsf{#1}}\fi
\ifx \bchapter \undefined \def \bchapter#1{#1}\fi
\ifx \bbook \undefined \def \bbook#1{#1}\fi
\ifx \bcomment \undefined \def \bcomment#1{#1}\fi
\ifx \oauthor \undefined \def \oauthor#1{#1}\fi
\ifx \citeauthoryear \undefined \def \citeauthoryear#1{#1}\fi
\ifx \endbibitem  \undefined \def \endbibitem {}\fi
\ifx \bconflocation  \undefined \def \bconflocation#1{#1}\fi
\ifx \arxivurl  \undefined \def \arxivurl#1{\textsf{#1}}\fi
\csname PreBibitemsHook\endcsname

\bibitem[\protect\citeauthoryear{Prothero and
  Robinson}{1974}]{prothero1974stability}
\begin{barticle}
\bauthor{\bsnm{Prothero}, \binits{A.}},
\bauthor{\bsnm{Robinson}, \binits{A.}}:
\batitle{On the stability and accuracy of one-step methods for solving stiff
  systems of ordinary differential equations}.
\bjtitle{Mathematics of Computation}
\bvolume{28}(\bissue{125}),
\bfpage{145}--\blpage{162}
(\byear{1974})
\end{barticle}
\endbibitem

\bibitem[\protect\citeauthoryear{Frank et~al.}{1981}]{frank1981concept}
\begin{barticle}
\bauthor{\bsnm{Frank}, \binits{R.}},
\bauthor{\bsnm{Schneid}, \binits{J.}},
\bauthor{\bsnm{Ueberhuber}, \binits{C.W.}}:
\batitle{The concept of {B}-convergence}.
\bjtitle{SIAM Journal on Numerical Analysis}
\bvolume{18}(\bissue{5}),
\bfpage{753}--\blpage{780}
(\byear{1981})
\doiurl{10.1137/0718051}
\end{barticle}
\endbibitem

\bibitem[\protect\citeauthoryear{Dekker and Verwer}{1984}]{dekker1984stability}
\begin{botherref}
\oauthor{\bsnm{Dekker}, \binits{K.}},
\oauthor{\bsnm{Verwer}, \binits{J.G.}}:
Stability of {Runge--Kutta} methods for stiff nonlinear differential equations.
CWI monographs
\textbf{2}
(1984)
\end{botherref}
\endbibitem

\bibitem[\protect\citeauthoryear{Scholz}{1989}]{scholz1989order}
\begin{barticle}
\bauthor{\bsnm{Scholz}, \binits{S.}}:
\batitle{Order barriers for the {B}-convergence of {ROW} methods}.
\bjtitle{Computing}
\bvolume{41}(\bissue{3}),
\bfpage{219}--\blpage{235}
(\byear{1989})
\doiurl{10.1007/BF02259094}
\end{barticle}
\endbibitem

\bibitem[\protect\citeauthoryear{Frank et~al.}{1985}]{frank1985order}
\begin{barticle}
\bauthor{\bsnm{Frank}, \binits{R.}},
\bauthor{\bsnm{Schneid}, \binits{J.}},
\bauthor{\bsnm{Ueberhuber}, \binits{C.W.}}:
\batitle{Order results for implicit {Runge--Kutta} methods applied to stiff
  systems}.
\bjtitle{SIAM Journal on Numerical Analysis}
\bvolume{22}(\bissue{3}),
\bfpage{515}--\blpage{534}
(\byear{1985})
\doiurl{10.1137/0722031}
\end{barticle}
\endbibitem

\bibitem[\protect\citeauthoryear{Burrage and
  Hundsdorfer}{1987}]{burrage1987order}
\begin{barticle}
\bauthor{\bsnm{Burrage}, \binits{K.}},
\bauthor{\bsnm{Hundsdorfer}, \binits{W.H.}}:
\batitle{The order of {B}-convergence of algebraically stable {Runge--Kutta}
  methods}.
\bjtitle{BIT Numerical Mathematics}
\bvolume{27}(\bissue{1}),
\bfpage{62}--\blpage{71}
(\byear{1987})
\doiurl{10.1007/BF01937355}
\end{barticle}
\endbibitem

\bibitem[\protect\citeauthoryear{Hairer and Wanner}{1996}]{hairer1996solving}
\begin{bbook}
\bauthor{\bsnm{Hairer}, \binits{E.}},
\bauthor{\bsnm{Wanner}, \binits{G.}}:
\bbtitle{Solving Ordinary Differential Equations II: Stiff and
  Differential-Algebraic Problems},
\bedition{2}nd edn.
\bsertitle{Springer Series in Computational Mathematics},
vol. \bseriesno{14}.
\bpublisher{Springer},
\blocation{Berlin, Heidelberg}
(\byear{1996})
\end{bbook}
\endbibitem

\bibitem[\protect\citeauthoryear{Burrage et~al.}{1986}]{burrage1986study}
\begin{barticle}
\bauthor{\bsnm{Burrage}, \binits{K.}},
\bauthor{\bsnm{Hundsdorfer}, \binits{W.H.}},
\bauthor{\bsnm{Verwer}, \binits{J.G.}}:
\batitle{A study of {B}-convergence of {Runge--Kutta} methods}.
\bjtitle{Computing}
\bvolume{36}(\bissue{1}),
\bfpage{17}--\blpage{34}
(\byear{1986})
\doiurl{10.1007/BF02238189}
\end{barticle}
\endbibitem

\bibitem[\protect\citeauthoryear{Auzinger et~al.}{1992}]{auzinger1992extension}
\begin{barticle}
\bauthor{\bsnm{Auzinger}, \binits{W.}},
\bauthor{\bsnm{Frank}, \binits{R.}},
\bauthor{\bsnm{Kirlinger}, \binits{G.}}:
\batitle{An extension of {B}-convergence for {Runge--Kutta} methods}.
\bjtitle{Applied Numerical Mathematics}
\bvolume{9}(\bissue{2}),
\bfpage{91}--\blpage{109}
(\byear{1992})
\doiurl{10.1016/0168-9274(92)90008-2}
\end{barticle}
\endbibitem

\bibitem[\protect\citeauthoryear{Calvo et~al.}{2000}]{calvo2000runge}
\begin{barticle}
\bauthor{\bsnm{Calvo}, \binits{M.}},
\bauthor{\bsnm{Gonz{\'a}lez-Pinto}, \binits{S.}},
\bauthor{\bsnm{Montijano}, \binits{J.I.}}:
\batitle{{R}unge--{K}utta methods for the numerical solution of stiff
  semilinear systems}.
\bjtitle{BIT Numerical Mathematics}
\bvolume{40}(\bissue{4}),
\bfpage{611}--\blpage{639}
(\byear{2000})
\doiurl{10.1023/A:1022332200092}
\end{barticle}
\endbibitem

\bibitem[\protect\citeauthoryear{Strehmel and Weiner}{1987}]{strehmel1987b}
\begin{barticle}
\bauthor{\bsnm{Strehmel}, \binits{K.}},
\bauthor{\bsnm{Weiner}, \binits{R.}}:
\batitle{{B}-convergence results for linearly implicit one step methods}.
\bjtitle{BIT Numerical Mathematics}
\bvolume{27}(\bissue{2}),
\bfpage{264}--\blpage{281}
(\byear{1987})
\doiurl{10.1007/BF01934189}
\end{barticle}
\endbibitem

\bibitem[\protect\citeauthoryear{Skvortsov}{2003}]{skvortsov2003accuracy}
\begin{barticle}
\bauthor{\bsnm{Skvortsov}, \binits{L.M.}}:
\batitle{Accuracy of {Runge--Kutta} methods applied to stiff problems}.
\bjtitle{Computational Mathematics and Mathematical Physics}
\bvolume{43}(\bissue{9}),
\bfpage{1320}--\blpage{1330}
(\byear{2003})
\end{barticle}
\endbibitem

\bibitem[\protect\citeauthoryear{Skvortsov}{2010}]{skvortsov2010model}
\begin{barticle}
\bauthor{\bsnm{Skvortsov}, \binits{L.M.}}:
\batitle{Model equations for accuracy investigation of {Runge--Kutta} methods}.
\bjtitle{Mathematical Models and Computer Simulations}
\bvolume{2}(\bissue{6}),
\bfpage{800}--\blpage{811}
(\byear{2010})
\doiurl{10.1134/S2070048210060165}
\end{barticle}
\endbibitem

\bibitem[\protect\citeauthoryear{Cai et~al.}{2025}]{CaiWanKareem2025}
\begin{barticle}
\bauthor{\bsnm{Cai}, \binits{Y.}},
\bauthor{\bsnm{Wan}, \binits{J.}},
\bauthor{\bsnm{Kareem}, \binits{A.}}:
\batitle{On convergence of implicit {R}unge-{K}utta methods for the
  incompressible {N}avier-{S}tokes equations with unsteady inflow}.
\bjtitle{Journal of Computational Physics}
\bvolume{523},
\bfpage{113627}
(\byear{2025})
\doiurl{10.1016/j.jcp.2024.113627}
\end{barticle}
\endbibitem

\bibitem[\protect\citeauthoryear{Hochbruck and
  Ostermann}{2005}]{hochbruck2005explicit}
\begin{barticle}
\bauthor{\bsnm{Hochbruck}, \binits{M.}},
\bauthor{\bsnm{Ostermann}, \binits{A.}}:
\batitle{Explicit exponential {Runge--Kutta} methods for semilinear parabolic
  problems}.
\bjtitle{SIAM Journal on Numerical Analysis}
\bvolume{43}(\bissue{3}),
\bfpage{1069}--\blpage{1090}
(\byear{2005})
\doiurl{10.1137/040611434}
\end{barticle}
\endbibitem

\bibitem[\protect\citeauthoryear{Luan and
  Ostermann}{2013}]{luan2013exponential}
\begin{barticle}
\bauthor{\bsnm{Luan}, \binits{V.T.}},
\bauthor{\bsnm{Ostermann}, \binits{A.}}:
\batitle{Exponential {B}-series: The stiff case}.
\bjtitle{SIAM Journal on Numerical Analysis}
\bvolume{51}(\bissue{6}),
\bfpage{3431}--\blpage{3445}
(\byear{2013})
\doiurl{10.1137/130920204}
\end{barticle}
\endbibitem

\bibitem[\protect\citeauthoryear{Hochbruck
  et~al.}{2020}]{hochbruck2020convergence}
\begin{barticle}
\bauthor{\bsnm{Hochbruck}, \binits{M.}},
\bauthor{\bsnm{Leibold}, \binits{J.}},
\bauthor{\bsnm{Ostermann}, \binits{A.}}:
\batitle{On the convergence of {L}awson methods for semilinear stiff problems}.
\bjtitle{Numerische Mathematik}
\bvolume{145}(\bissue{3}),
\bfpage{553}--\blpage{580}
(\byear{2020})
\doiurl{10.1007/s00211-020-01120-4}
\end{barticle}
\endbibitem

\bibitem[\protect\citeauthoryear{Hansen and Ostermann}{2016}]{hansen2016high}
\begin{barticle}
\bauthor{\bsnm{Hansen}, \binits{E.}},
\bauthor{\bsnm{Ostermann}, \binits{A.}}:
\batitle{High-order splitting schemes for semilinear evolution equations}.
\bjtitle{BIT Numerical Mathematics}
\bvolume{56}(\bissue{4}),
\bfpage{1303}--\blpage{1316}
(\byear{2016})
\doiurl{10.1007/s10543-016-0604-2}
\end{barticle}
\endbibitem

\bibitem[\protect\citeauthoryear{Einkemmer and
  Ostermann}{2015}]{einkemmer2015overcoming}
\begin{barticle}
\bauthor{\bsnm{Einkemmer}, \binits{L.}},
\bauthor{\bsnm{Ostermann}, \binits{A.}}:
\batitle{Overcoming order reduction in diffusion-reaction splitting. part 1:
  {D}irichlet boundary conditions}.
\bjtitle{SIAM Journal on Scientific Computing}
\bvolume{37}(\bissue{3}),
\bfpage{1577}--\blpage{1592}
(\byear{2015})
\doiurl{10.1137/140994204}
\end{barticle}
\endbibitem

\bibitem[\protect\citeauthoryear{Einkemmer and
  Ostermann}{2016}]{einkemmer2016overcoming}
\begin{barticle}
\bauthor{\bsnm{Einkemmer}, \binits{L.}},
\bauthor{\bsnm{Ostermann}, \binits{A.}}:
\batitle{Overcoming order reduction in diffusion-reaction splitting. part 2:
  Oblique boundary conditions}.
\bjtitle{SIAM Journal on Scientific Computing}
\bvolume{38}(\bissue{6}),
\bfpage{3741}--\blpage{3757}
(\byear{2016})
\doiurl{10.1137/16M1056250}
\end{barticle}
\endbibitem

\bibitem[\protect\citeauthoryear{Lubich and
  Ostermann}{1995}]{lubich1995linearly}
\begin{barticle}
\bauthor{\bsnm{Lubich}, \binits{C.}},
\bauthor{\bsnm{Ostermann}, \binits{A.}}:
\batitle{{Linearly implicit time discretization of non-linear parabolic
  equations}}.
\bjtitle{IMA Journal of Numerical Analysis}
\bvolume{15}(\bissue{4}),
\bfpage{555}--\blpage{583}
(\byear{1995})
\doiurl{10.1093/imanum/15.4.555}
\end{barticle}
\endbibitem

\bibitem[\protect\citeauthoryear{Ketcheson et~al.}{2020}]{ketcheson2020dirk}
\begin{bchapter}
\bauthor{\bsnm{Ketcheson}, \binits{D.I.}},
\bauthor{\bsnm{Seibold}, \binits{B.}},
\bauthor{\bsnm{Shirokoff}, \binits{D.}},
\bauthor{\bsnm{Zhou}, \binits{D.}}:
\bctitle{{DIRK} schemes with high weak stage order}.
In: \beditor{\bsnm{Sherwin}, \binits{S.J.}},
\beditor{\bsnm{Moxey}, \binits{D.}},
\beditor{\bsnm{Peir{\'o}}, \binits{J.}},
\beditor{\bsnm{Vincent}, \binits{P.E.}},
\beditor{\bsnm{Schwab}, \binits{C.}} (eds.)
\bbtitle{Spectral and High Order Methods for Partial Differential Equations
  ICOSAHOM 2018},
pp. \bfpage{453}--\blpage{463}.
\bpublisher{Springer},
\blocation{Cham}
(\byear{2020})
\end{bchapter}
\endbibitem

\bibitem[\protect\citeauthoryear{Albrecht}{1987}]{albrecht1987new}
\begin{barticle}
\bauthor{\bsnm{Albrecht}, \binits{P.}}:
\batitle{A new theoretical approach to {R}unge--{K}utta methods}.
\bjtitle{SIAM Journal on Numerical Analysis}
\bvolume{24}(\bissue{2}),
\bfpage{391}--\blpage{406}
(\byear{1987})
\doiurl{10.1137/0724030}
\end{barticle}
\endbibitem

\bibitem[\protect\citeauthoryear{Albrecht}{1996}]{albrecht1996runge}
\begin{barticle}
\bauthor{\bsnm{Albrecht}, \binits{P.}}:
\batitle{The {R}unge--{K}utta theory in a nutshell}.
\bjtitle{SIAM Journal on Numerical Analysis}
\bvolume{33}(\bissue{5}),
\bfpage{1712}--\blpage{1735}
(\byear{1996})
\doiurl{10.1137/S0036142994260872}
\end{barticle}
\endbibitem

\bibitem[\protect\citeauthoryear{Boscarino and
  Russo}{2009}]{BoscarinoRusso2009}
\begin{barticle}
\bauthor{\bsnm{Boscarino}, \binits{S.}},
\bauthor{\bsnm{Russo}, \binits{G.}}:
\batitle{On a class of uniformly accurate {IMEX} {R}unge-{K}utta schemes and
  applications to hyperbolic systems with relaxation}.
\bjtitle{SIAM Journal on Scientific Computing}
\bvolume{31}(\bissue{3}),
\bfpage{1926}--\blpage{1945}
(\byear{2009})
\end{barticle}
\endbibitem

\bibitem[\protect\citeauthoryear{Hu and Shu}{2025}]{HuShu2025}
\begin{barticle}
\bauthor{\bsnm{Hu}, \binits{J.}},
\bauthor{\bsnm{Shu}, \binits{R.}}:
\batitle{Uniform accuracy of implicit-explicit {R}unge-{K}utta ({IMEX-RK})
  schemes for hyperbolic systems with relaxation}.
\bjtitle{Math. Comp.}
\bvolume{94},
\bfpage{209}--\blpage{240}
(\byear{2025})
\end{barticle}
\endbibitem

\bibitem[\protect\citeauthoryear{Luan and Ostermann}{2014}]{LuanOstermann2014}
\begin{barticle}
\bauthor{\bsnm{Luan}, \binits{V.T.}},
\bauthor{\bsnm{Ostermann}, \binits{A.}}:
\batitle{Exponential {R}osenbrock methods of order five — construction,
  analysis and numerical comparisons}.
\bjtitle{Journal of Computational and Applied Mathematics}
\bvolume{255},
\bfpage{417}--\blpage{431}
(\byear{2014})
\end{barticle}
\endbibitem

\bibitem[\protect\citeauthoryear{Nevanlinna}{1985}]{nevanlinna1985matrix}
\begin{barticle}
\bauthor{\bsnm{Nevanlinna}, \binits{O.}}:
\batitle{Matrix valued versions of a result of von {N}eumann with an
  application to time discretization}.
\bjtitle{Journal of Computational and Applied Mathematics}
\bvolume{12-13},
\bfpage{475}--\blpage{489}
(\byear{1985})
\doiurl{10.1016/0377-0427(85)90041-X}
\end{barticle}
\endbibitem

\bibitem[\protect\citeauthoryear{Taylor}{2011}]{TaylorPDE1}
\begin{bbook}
\bauthor{\bsnm{Taylor}, \binits{M.E.}}:
\bbtitle{Partial Differential Equations I},
\bedition{2nd} edn.
\bpublisher{Springer},
\blocation{Cham}
(\byear{2011})
\end{bbook}
\endbibitem

\bibitem[\protect\citeauthoryear{Rang}{2014}]{rang2014analysis}
\begin{barticle}
\bauthor{\bsnm{Rang}, \binits{J.}}:
\batitle{An analysis of the {Prothero--Robinson} example for constructing new
  {DIRK} and {ROW} methods}.
\bjtitle{Journal of Computational and Applied Mathematics}
\bvolume{262},
\bfpage{105}--\blpage{114}
(\byear{2014})
\doiurl{10.1016/j.cam.2013.09.062}
\end{barticle}
\endbibitem

\bibitem[\protect\citeauthoryear{Ostermann and
  Roche}{1992}]{ostermann1992runge}
\begin{barticle}
\bauthor{\bsnm{Ostermann}, \binits{A.}},
\bauthor{\bsnm{Roche}, \binits{M.}}:
\batitle{{Runge--Kutta} methods for partial differential equations and
  fractional orders of convergence}.
\bjtitle{Mathematics of Computation}
\bvolume{59}(\bissue{200}),
\bfpage{403}--\blpage{420}
(\byear{1992})
\doiurl{10.1090/s0025-5718-1992-1142285-6}
\end{barticle}
\endbibitem

\bibitem[\protect\citeauthoryear{Rosales et~al.}{2024}]{rosales2024spatial}
\begin{barticle}
\bauthor{\bsnm{Rosales}, \binits{R.R.}},
\bauthor{\bsnm{Seibold}, \binits{B.}},
\bauthor{\bsnm{Shirokoff}, \binits{D.}},
\bauthor{\bsnm{Zhou}, \binits{D.}}:
\batitle{Spatial manifestations of order reduction in {R}unge-{K}utta methods
  for initial boundary value problems}.
\bjtitle{Commun. Math. Sci.}
\bvolume{22}(\bissue{3}),
\bfpage{613}--\blpage{653}
(\byear{2024})
\doiurl{10.4310/CMS.2024.v22.n3.a2}
\end{barticle}
\endbibitem

\bibitem[\protect\citeauthoryear{Biswas et~al.}{2024}]{biswas2022algebraic}
\begin{barticle}
\bauthor{\bsnm{Biswas}, \binits{A.}},
\bauthor{\bsnm{Ketcheson}, \binits{D.I.}},
\bauthor{\bsnm{Seibold}, \binits{B.}},
\bauthor{\bsnm{Shirokoff}, \binits{D.}}:
\batitle{Algebraic structure of the weak stage order conditions for
  {R}unge--{K}utta methods}.
\bjtitle{SIAM Journal on Numerical Analysis}
\bvolume{62}(\bissue{1}),
\bfpage{48}--\blpage{72}
(\byear{2024})
\doiurl{10.1137/22M1483943}
\end{barticle}
\endbibitem

\bibitem[\protect\citeauthoryear{Biswas et~al.}{2023}]{biswas2023design}
\begin{barticle}
\bauthor{\bsnm{Biswas}, \binits{A.}},
\bauthor{\bsnm{Ketcheson}, \binits{D.I.}},
\bauthor{\bsnm{Seibold}, \binits{B.}},
\bauthor{\bsnm{Shirokoff}, \binits{D.}}:
\batitle{Design of {DIRK} schemes with high weak stage order}.
\bjtitle{Communications in Applied Mathematics and Computational Science}
\bvolume{18}(\bissue{1}),
\bfpage{1}--\blpage{28}
(\byear{2023})
\end{barticle}
\endbibitem

\bibitem[\protect\citeauthoryear{Biswas et~al.}{2025}]{biswas2023explicit}
\begin{barticle}
\bauthor{\bsnm{Biswas}, \binits{A.}},
\bauthor{\bsnm{Ketcheson}, \binits{D.I.}},
\bauthor{\bsnm{Roberts}, \binits{S.}},
\bauthor{\bsnm{Seibold}, \binits{B.}},
\bauthor{\bsnm{Shirokoff}, \binits{D.}}:
\batitle{Explicit {Runge--Kutta} methods that alleviate order reduction}.
\bjtitle{SIAM Journal on Numerical Analysis}
\bvolume{63}(\bissue{4}),
\bfpage{1398}--\blpage{1426}
(\byear{2025})
\doiurl{10.1137/23M1606812}
\end{barticle}
\endbibitem

\bibitem[\protect\citeauthoryear{{OEIS Foundation Inc.}}{2025}]{oeis}
\begin{botherref}
\oauthor{\bsnm{{OEIS Foundation Inc.}}}:
The {O}n-{L}ine {E}ncyclopedia of {I}nteger {S}equences.
Published electronically at \url{http://oeis.org}
(2025)
\end{botherref}
\endbibitem

\bibitem[\protect\citeauthoryear{Hundsdorfer and
  Verwer}{2003}]{hundsdorfer2003numerical}
\begin{bbook}
\bauthor{\bsnm{Hundsdorfer}, \binits{W.H.}},
\bauthor{\bsnm{Verwer}, \binits{J.G.}}:
\bbtitle{Numerical Solution of Time-dependent Advection-diffusion-reaction
  Equations}
vol. \bseriesno{33}.
\bpublisher{Springer},
\blocation{Berlin/Heidelberg}
(\byear{2003})
\end{bbook}
\endbibitem

\bibitem[\protect\citeauthoryear{Hairer et~al.}{1982}]{hairer1982stability}
\begin{barticle}
\bauthor{\bsnm{Hairer}, \binits{E.}},
\bauthor{\bsnm{Bader}, \binits{G.}},
\bauthor{\bsnm{Lubich}, \binits{C.}}:
\batitle{On the stability of semi-implicit methods for ordinary differential
  equations}.
\bjtitle{BIT Numerical Mathematics}
\bvolume{22}(\bissue{2}),
\bfpage{211}--\blpage{232}
(\byear{1982})
\doiurl{10.1007/BF01944478}
\end{barticle}
\endbibitem

\bibitem[\protect\citeauthoryear{Hairer et~al.}{2006}]{hairerlubichwanner2006}
\begin{bbook}
\bauthor{\bsnm{Hairer}, \binits{E.}},
\bauthor{\bsnm{Lubich}, \binits{C.}},
\bauthor{\bsnm{Wanner}, \binits{G.}}:
\bbtitle{Geometric Numerical Integration},
\bedition{2}nd edn.
\bsertitle{Springer Series in Computational Mathematics},
vol. \bseriesno{31}.
\bpublisher{Springer},
\blocation{Dordrecht}
(\byear{2006}).
\doiurl{10.1007/3-540-30666-8}
\end{bbook}
\endbibitem

\end{thebibliography}

\end{document}